\documentclass[12pt]{amsart}
\usepackage{amsthm, amsmath, amssymb, mathrsfs, enumitem, svn, amsfonts, stmaryrd, mathtools, tikz-cd,mathscinet, amscd}
\usepackage{verbatim, comment} 
 \usepackage{bm}

\usepackage[a4paper,bindingoffset=0.2in,           left=0.6in,right=0.6in,top=0.6in,bottom=1in,            footskip=.25in]{geometry}

\usepackage{tabu} 
\usepackage[all, cmtip]{xy}

\usepackage{color}
\definecolor{ghcolor}{RGB}{0, 150, 200} 
\definecolor{winestain}{rgb}{0.5,0,0}
\usepackage[linktocpage=true, hidelinks, colorlinks, linkcolor=blue, citecolor=red, bookmarksopen=true, urlcolor=brown]{hyperref} 

\newtheorem{theorem}[subsubsection]{Theorem}
\newtheorem{thm}[subsubsection]{Theorem}
\newtheorem{lemma}[subsubsection]{Lemma}
\newtheorem{lem}[subsubsection]{Lemma}
\newtheorem{cor}[subsubsection]{Corollary}

\newtheorem{prop}[subsubsection]{Proposition}

\theoremstyle{definition}
\newtheorem{defn}[subsubsection]{Definition}
\newtheorem{construction}[subsubsection]{Construction}

\newtheorem{notation}[subsubsection]{Notation}
\newtheorem{Notation}[subsubsection]{Notation}
\newtheorem{convention}[subsubsection]{Convention}

\newtheorem{remark}[subsubsection]{Remark}
\newtheorem{rem}[subsubsection]{Remark}
\newtheorem{para}[subsubsection]{}
\numberwithin{equation}{subsection}
\def \into {\hookrightarrow }
\def \to {\rightarrow}

\newcommand{\Vect}{\mathrm{Vect}}

\newcommand{\GL}{\mathrm{GL}}

\def\Mat{\mathrm{Mat}}

\DeclareMathOperator{\rep}{Rep}

\renewcommand{\hom}{\mathrm{Hom}}

\DeclareMathOperator{\gal}{Gal}

\def\inf{{\mathrm{inf}}}
\def\sup{\mathrm{sup}}

\def\an{\mathrm{an}}
\def\perf{\mathrm{perf}}

\newcommand{\Lie}{\mathrm{Lie}}

\newcommand{\Kpinfty}{{K_{p^\infty}}}
\newcommand{\kpinfty}{{K_{p^\infty}}}
\newcommand{\hatkpinfty}{{\widehat{K_{p^\infty}}}}
\newcommand{\Kinfty}{{K_{\infty}}}
\newcommand{\kinfty}{{K_{\infty}}}

\newcommand{\gammak}{{\Gamma_K}}

\newcommand{\gk}{{G_K}}

\newcommand{\la}{{\mathrm{la}}}
\newcommand{\dan}{\text{$\mbox{-}\mathrm{an}$}}

\newcommand{\dla}{\text{$\mbox{-}\mathrm{la}$}}

\newcommand{\dpa}{\text{$\mbox{-}\mathrm{pa}$}}

\def\crys{{\mathrm{crys}}}

\def\Sen{{\mathrm{Sen}}}

\def\rig{{\mathrm{rig}}}

\newcommand{\ainf}{{\mathbf{A}_{\mathrm{inf}}}}

\newcommand{\bcrisplus}{{\mathbf{B}^+_{\mathrm{cris}}}}

\newcommand{\bdrplus}{{\mathbf{B}^+_{\mathrm{dR}}}}

 \newcommand{\A}{ {\mathbf{A}}   }
\newcommand{\B}{  {\mathbf{B}}  }

\newcommand{\wtb}{   {\widetilde{{\mathbf{B}}}}  }

 \def \O {{\mathcal{O}}}

\def \ok {{\mathcal{O}_K}}

\def \oc {{\mathcal{O}_C}}

\newcommand*{\wt}[1]{\widetilde{#1}}
\newcommand*{\wh}[1]{\widehat{#1}}

\newcommand{\Zp}{{\mathbb{Z}_p}}
\newcommand{\Qp}{{\mathbb{Q}_p}}

\newcommand{\zp}{{\mathbb{Z}_p}}
\newcommand{\qp}{{\mathbb{Q}_p}}

   \def \calO {{\mathcal{O}}}

\newcommand{\gs}{{\mathfrak{S}}}

\newcommand{\fkm}{{\mathfrak{m}}}

\newcommand{\fkt}{{\mathfrak{t}}}

\newcommand{\cbf}{\mathbf{c}}

\newcommand{\kbf}{\mathbf{k}}

\usepackage{relsize}
\usepackage[bbgreekl]{mathbbol}
\usepackage{amsfonts}

\DeclareSymbolFontAlphabet{\mathbb}{AMSb}
\DeclareSymbolFontAlphabet{\mathbbl}{bbold}
\newcommand{\prism}{{\mathlarger{\mathbbl{\Delta}}}}
\newcommand{\Prism}{{\mathlarger{\mathbbl{\Delta}}}}

\newcommand{\oprism}{{\mathcal{O}_\prism}}
\newcommand{\okprism}{{(\ok)_\prism}}
\newcommand{\baroprism}{{\overline{\O}_\prism}}

\newcommand{\ocflat}{{\mathcal{O}_C^\flat}}

\newcommand{\smat}[1]{\left( \begin{smallmatrix} #1 \end{smallmatrix} \right)}
\newcommand{\dacc}[1]{\{\!\{ #1 \}\!\}}
\newcommand{\bM}{{\mathbb M}}
\renewcommand{\bm}{{\mathbb M}}

\begin{document}
\title[]{Hodge-Tate prismatic crystals and Sen theory}

\date{\today} 
\author[]{Hui Gao}   \address{Department of Mathematics, Southern University of Science and Technology, Shenzhen 518055,  China}   \email{gaoh@sustech.edu.cn}

\begin{abstract}
Let $K$ be a mixed characteristic complete discrete valuation field with perfect residue field, and let $\kinfty/K$ be a Kummer tower extension by adjoining   a compatible system of $p$-power roots of a chosen uniformizer. 
We use this Kummer tower to reconstruct Sen theory which classically is  obtained using the cyclotomic tower. Using this Sen theory over the Kummer tower, we prove a conjecture of Min-Wang which predicts that    Hodge-Tate prismatic crystals are determined by the Sen operator; this implies that the category of  (rational)  Hodge-Tate prismatic crystals is equivalent to the category of  nearly Hodge-Tate representations.
\end{abstract}

\maketitle
\tableofcontents

\section{Introduction}\label{secintro}
  The   goal of this paper is to study Hodge-Tate prismatic crystals, and the main result is stated in  \S \ref{s11}.
 We  explain how to reduce the prismatic theorem  to a more ``classical" (non-prismatic) problem, which in turn will be addressed in \S \ref{s12} using   Sen theory over the Kummer tower.

\subsection{Prismatic site and prismatic crystals}\label{s11}

In \cite{BSprism}, Bhatt-Scholze introduce  the  prismatic site and use  it to recover all known \emph{integral} $p$-adic cohomology theories, including the ones constructed in the work of  Bhatt-Morrow-Scholze \cite{BMS1, BMS2}.
The cohomology theories in \cite{BSprism} are defined over the \emph{relative} prismatic site, where one only considers prisms over a \emph{fixed} prism. 
Nonetheless, an \emph{absolute} prismatic site  is also defined and used therein.
It turns out this absolute prismatic site is more suitable to study ``arithmetic" problems. 
For example,  Ansch\"utz-Le Bras \cite{ALB} show that for a ($\zp$-flat) quasi-syntomic ring $R$, filtered prismatic Dieudonn\'e crystals   on $R_\prism$ (the absolute prismatic site of $R$) classify $p$-divisible groups over $R$. 
Recently, Bhatt-Scholze \cite{BSFcrystal} show that for $\ok$  a mixed characteristic complete discrete valuation ring with perfect residue field,   prismatic $F$-crystals on $(\ok)_\prism$ classify integral crystalline Galois representations.
In this paper, we study   \emph{Hodge-Tate prismatic crystals} on $(\ok)_\prism$.
Let us now quickly set up some notations to facilitate our discussions.

\begin{notation}   \label{notaprism}
\begin{enumerate}
\item Let $\ok$ be a mixed characteristic complete discrete valuation ring with perfect residue field $k$. Let $K$ be its fraction field; fix an algebraic closure $\overline{K}$ and denote $\gk=\gal(\overline{K}/K)$. 
 Let $\rep_\gk(\zp)$ be the category of finite free (continuous) $\zp$-representations of $\gk$, and let $\mathrm{Rep}^{\crys}_{\gk}(\zp)$ be the sub-category of integral crystalline representations.

\item  Let $\okprism$ denote  the absolute  prismatic site of $\ok$. It  is the opposite of the category of bounded prisms $(A,I)$ together with a map $\ok \to A/I$, endowed with the Grothendieck topology for which covers are morphisms of prisms $(A, I) \to (B,J)$, such that the underlying ring map $A\to B$ is $(p,I)$-completely faithfully flat. 
Let $(\calO_K)^{\perf}_{\Prism}$ be the sub-site consisting of perfect prisms.

\item 
Let $\mathcal{O}_\Prism$ be the  {structure sheaf} on $\okprism$, so that $\mathcal{O}_\Prism((A, I))=A.$
We    define the  sheaf  $\mathcal{I}_\prism$ resp. $\baroprism$ so that 
$$\mathcal{I}_\prism((A, I))  = I \text{ resp. } \baroprism((A, I)) = A/I.$$
Let $\mathcal{O}_\Prism[1/\mathcal{I}_\Prism]^{\wedge}_p$ be the $p$-adic completion of $\mathcal{O}_\Prism[1/\mathcal{I}_\Prism]$, and let $\baroprism[1/p]$ be the sheaf with $p$ inverted.
 See \cite[\S 2]{BSFcrystal} for more discussion about these sheaves.
 
\item 
 Let $\mathrm{Vect}(\okprism,\mathcal{O}_\Prism)$ denote the category of {\em prismatic  crystals (of vector bundles) on $\okprism$}. Namely, an object $\bM$ is a sheaf of $\oprism$-modules such that for any $(A,I)\in \okprism$, $\bM((A,I))$ is a finite projective $A$-module and that for any morphism $(A,I)\to (B,J)$, the natural map
    \[\bM((A,I))\otimes_A B\to \bM((B,J))\]
    is an isomorphism.
    Let $\mathrm{Vect}^{\varphi}((\calO_K)_{\Prism}, \mathcal{O}_\Prism)$ denote the category of  {\em prismatic  $F$-crystals (of vector bundles) on $\okprism$}, where an object is a prismatic  crystal $\bM$  equipped with an identification 
    $$\varphi_{\bM}: \varphi^* \bM[1/\mathcal{I}_\Prism] \simeq \bM[1/\mathcal{I}_\Prism].$$
   
\item We can analogously define:
\begin{itemize}
\item $\mathrm{Vect}^{\varphi}(\okprism, \mathcal{O}_\Prism[1/\mathcal{I}_\Prism]^{\wedge}_p)$, called the category of \emph{Laurent prismatic $F$-crystals}, cf. \cite[Def. 3.2, Ex. 4.4]{BSFcrystal}.
\item $\Vect((\calO_K)_{\Prism},\overline \calO_{\Prism}[ {1}/{p}])$, called the category of \emph{rational Hodge-Tate prismatic crystals}.
\item $\Vect((\calO_K)^{\perf}_{\Prism},\overline \calO_{\Prism}[ {1}/{p}])$, called the category of \emph{rational Hodge-Tate prismatic crystals on the perfect prismatic site}.
\end{itemize}
 Note the last two categories (defined in \cite{MW})   do not have $\varphi$-structures; cf. also \S \ref{sHT} for   detailed definitions. 
\end{enumerate}
\end{notation}

The main results of \cite{BSFcrystal} can be summarized by the following diagram (of tensor functors), where we use $\simeq$ resp. $\hookrightarrow$ to signify an equivalence resp. a fully faithful functor.
 \begin{equation}\label{diagbs21}
\begin{tikzcd}
{\mathrm{Vect}^{\varphi}((\calO_K)_{\Prism}, \mathcal{O}_\Prism)} \arrow[rr, hook] \arrow[d, "\simeq"] &  & {\mathrm{Vect}^{\varphi}(\okprism, \mathcal{O}_\Prism[1/\mathcal{I}_\Prism]^{\wedge}_p)} \arrow[d, "\simeq"] \\
\mathrm{Rep}^{\crys}_{\gk}(\zp) \arrow[rr, hook]                                                              &  & \rep_\gk(\zp)                                                                                               
\end{tikzcd}
\end{equation}
Here, the left vertical equivalence is their main theorem \cite[Thm. 1.2]{BSFcrystal}, and the right vertical equivalence is a special case of \cite[Cor. 3.7, Ex. 4.4]{BSFcrystal}.  
(Recently, this picture is generalized to the semi-stable case by Du-Liu \cite{DL21}, using the log-prismatic site of Koshikawa \cite{Kos21}.)

 Our main result Thm. \ref{thmintro} constructs an analogous diagram.
To state the theorem, we quickly recall Sen theory  \cite{Sen80}.
Let $C$ be the $p$-adic completion of $\overline{K}$, on which $G_K$ acts continuously.
Let $\rep_\gk(C)$ be the category of ``$C$-representations": an object is a finite dimensional $C$-vector space together with a continuous and \emph{semi-linear} action of $G_K$.
Let $\kpinfty$ be the (cyclotomic) extension of $K$ by adjoining all $p$-power roots of unity.
  To a $C$-representation, Sen    constructs a canonical  finite dimensional  $\kpinfty$-vector space equipped with a  linear operator;
this operator is   nowadays  called the \emph{Sen operator}, whose eigenvalues   are   called the (Hodge-Tate-)Sen weights   of the corresponding $C$-representation.

\begin{notation}\label{notaeu}
Let $W(k)$ be the ring of Witt vectors, and let $K_0=W(k)[1/p]$.
Let $\pi \in K$ be a \emph{fixed} uniformizer, and let $E(u)=\mathrm{Irr}(\pi, K_0) \in W(k)[u]$ be the minimal polynomial over $K_0$. Let $E'(u) =\frac{d}{du}E(u)$.
\end{notation}

\begin{defn}\label{defnnht}
Say $W\in \rep_\gk(C)$ is \emph{nearly Hodge-Tate}   if all of its  Sen weights are in the subset
\begin{equation}
\label{eqeigen}
\mathbb{Z} + (E'(\pi))^{-1}\cdot \mathfrak{m}_{\O_{\overline{K}}},   \end{equation}
 where $\O_{\overline{K}}$ is the ring of integers of $\overline{K}$ with $\mathfrak{m}_{\O_{\overline{K}}}$  its maximal ideal.
 That is: the Sen weights are near to being an integer up to a bounded distance. (Recall by \cite[Chap. III, \S 6, Cor. 2]{Serrelocal}, the ideal $E'(\pi)\cdot\ok$ is precisely the different $\mathfrak{D}_{\ok/W(k)}$ and hence is independent of choices of $\pi$; thus our definition is also independent of choices of $\pi$.)
Write $\rep_\gk^{\mathrm{nHT}}(C)$ for the (tensor) subcategory of $\rep_\gk(C)$ consisting of these objects.
\end{defn}

\begin{remark}
Let  $W\in \rep_\gk(C)$ and let $L/K$ be a finite extension. Then $W$ is nearly Hodge-Tate implies that the restriction $W|_{G_L}$ is also nearly Hodge-Tate, but not vice versa (unless $L/K$ is unramified).
\end{remark}



The following is our main theorem, which in particular  confirms a conjecture of Yu Min and Yupeng Wang \cite[Conj. 3.17]{MW}.

\begin{theorem}\label{thmintro}
We have a commutative diagram of tensor functors:
\begin{equation}\label{diagnht}
\begin{tikzcd}
{\Vect((\calO_K)_{\Prism},\overline \calO_{\Prism}[\frac{1}{p}])} \arrow[rr, hook] \arrow[d, "\simeq"] &  & {\Vect((\calO_K)^{\perf}_{\Prism},\overline \calO_{\Prism}[\frac{1}{p}])} \arrow[d, "\simeq"] \\
\rep_\gk^{\mathrm{nHT}}(C) \arrow[rr, hook]                                               &  & \rep_\gk(C)                                                                                  
\end{tikzcd}
\end{equation}
\end{theorem}

\begin{remark} 
We comment on the genesis of the theorem, and its relation with other works.
\begin{enumerate}
\item In the diagram \eqref{diagnht}, the right vertical equivalence  is already proved by Min-Wang \cite[Thm. 3.12]{MW}, cf. also  Prop. \ref{crystal is C-rep} in this paper for a review.
The left vertical equivalence,  proved in Thm. \ref{thmnht}, is a   consequence of a Conjecture of Min-Wang \cite[Conj. 3.17]{MW} which in turn is proved in Thm. \ref{thmMWSen}. 
We   remark that the notion of \emph{nearly Hodge-Tate representations}, although not explicitly mentioned in   \cite{MW}, is first discovered by Min-Wang (as a consequence of their conjecture just mentioned); the author thanks them for informing him this notion and for allowing him to discuss it in this paper.


\item The techniques and ideas in this paper can be generalized to the relative case (i.e., for a smooth formal scheme over $\ok$), for example, we can prove \cite[Conj. 3.21]{MW22} in the recent preprint of Min-Wang, which is generalization of \cite[Conj. 3.17]{MW} to the relative case. Details will appear elsewhere.
 
\item Hodge-Tate prismatic crystals (in the relative case) are also studied in the recent preprints of Bhatt-Lurie \cite{BL1} and \cite{BL2}, using (the Hodge-Tate divisor of) the \emph{Cartier-Witt stack}; in particular, relation  with Sen theory is discussed in \cite[\S 3.9]{BL1}. The exact connection with our work   is yet to be discovered.
\end{enumerate}
\end{remark}

 \begin{rem} \label{remintro115}
 We make some more general remarks about the main theorem.
 \begin{enumerate}
 \item 
Recall $W\in \rep_\gk(C)$ is called \emph{almost Hodge-Tate} (cf. \cite{Fon04}) if its Sen weights are integers, and is  called \emph{Hodge-Tate} if  additionally the Sen operator is    semi-simple. 
Thus, the nearly Hodge-Tate representations are those \emph{near} to an almost Hodge-Tate representation. 
The similarities between the diagrams \eqref{diagbs21} and \eqref{diagnht} seem to suggest that the  nearly Hodge-Tate representations could be regarded as the ``crystalline objects" (whatever that means) in the category of $C$-representations; this analogy deserves further studies.


\item 
If one changes the two appearances of $\overline \calO_{\Prism}[\frac{1}{p}]$ in Diagram \eqref{diagnht} to a ``prismatic de Rham sheaf" (cf. \cite[Construction 6.4]{BSFcrystal} for a related discussion), it might be possible to introduce the ``nearly de Rham representations"; recall \emph{almost  de Rham representations} are discussed in \cite{Fon04}. 
Nonetheless, we expect this picture to be substantially more complicated: indeed, the Sen operator is only a linear operator, but the Sen-Fontaine operator in \cite{Fon04} is a \emph{differential} operator.
We hope to come back to this direction in the near future.


\item In \cite{MW}, they also give precise classification of the \emph{integral} Hodge-Tate prismatic crystals via (integral) stratifications, cf. \S \ref{s41ht}. It remains a curious question that if one can obtain an \emph{integral} version of diagram \eqref{diagnht}: namely, remove $1/p$ on the top, and change $C$ to $\O_C$ in the bottom. This  (at least) would entail a substantial (integrality) checking of many results, including those in the classical paper \cite{Sen80}. We do not pursue this here.

 \end{enumerate}
 \end{rem}

We now explain how to reduce the above ``prismatic" theorem to a more ``classical" problem close to the context in the study of (overconvergent) $(\varphi, \tau)$-modules. We first recall the   two (most) important prisms.

\begin{notation}\label{notaprism}
\begin{enumerate}
\item Let $\gs=W(k)[[u]]$, and equip it with a Frobenius $\varphi$ extending  the absolute Frobenius on $W(k)$ and such that $\varphi(u)=u^p$. Then $(\gs, (E(u)) \in \okprism$, and is called the Breuil-Kisin prism (associated to $\pi$).
\item Recall $C$ is the $p$-adic completion of $\overline K$, let $C^\flat$ be the tilt with ring of integers $\O_{C^\flat}$. Let $\ainf=W(\O_{C^\flat})$ be the ring of Witt vectors, equipped with the absolute Frobenius. There is a usual Fontaine's map $\theta: \ainf \to \O_C$ whose kernel  principal. Then $(\ainf, \ker\theta) \in \okprism$, and is called the Fontaine prism.
\item Let $\pi_0=\pi$, and for each $n \geq 1$, inductive choose some $\pi_n$ so that $\pi_n^p=\pi_{n-1}$. This compatible sequence defines an element $\underline{\pi} \in \ocflat$. We can define a morphism of prisms
$$(\gs, (E)) \to (\ainf, \ker\theta)$$
which is a $W(k)$-linear map sends $u$   to the Teichm\"uller lift $[\underline{\pi}]$.

\item We further introduce some field notations. Let $\mu_1$ be a primitive $p$-root of unity, and inductively, for each $n \geq 2$, choose $\mu_n$ a $p$-th root of $\mu_{n-1}$. Define the fields
$$K_{\infty}   = \cup _{n = 1} ^{\infty} K(\pi_n), \quad K_{p^\infty}=  \cup _{n=1}^\infty
K(\mu_{n}), \quad L =  \cup_{n = 1} ^{\infty} K(\pi_n, \mu_n).$$
Let $$G_{\kinfty}:= \gal (\overline K / K_{\infty}), \quad G_{\kpinfty}:= \gal (\overline K / K_{p^\infty}), \quad G_L: =\gal(\overline K/L).$$
Further define $\Gamma_K, \hat{G}$ as in the following diagram, where we let $\tau$ be a topological generator of $\gal(L/\kpinfty) \simeq \zp$, cf. Notation \ref{nota hatG} for more details.
\begin{equation}
\begin{tikzcd}
                                       & L                                                                                             &                             \\
\kpinfty \arrow[ru, "<\tau>", no head] &                                                                                               & \kinfty \arrow[lu, no head] \\
                                       & K \arrow[lu, "\Gamma_K", no head] \arrow[ru, no head] \arrow[uu, "\hat{G}"', no head, dashed] &                            
\end{tikzcd}
\end{equation}
\end{enumerate}
\end{notation}

Coming back to Thm. \ref{thmintro}, we focus on explaining the left vertical equivalence. In fact, the difficulty lies in constructing the functor; once constructed, one then uses results from  \cite{Sen80} to show it is an equivalence.
Indeed, let $\bM \in  \Vect((\calO_K)_{\Prism},\overline \calO_{\Prism}[\frac{1}{p}])$, and consider the  evaluation 
$$W:=\bM((\ainf, \ker\theta)) \in \rep_\gk(C);$$ it suffices to show that
  $W$ is nearly Hodge-Tate.
We now use the Breuil-Kisin prism $(\gs, (E))$, which is a cover of the final object in $\mathrm{Shv}(\okprism)$. Thus $\bM$ is completely determined by the evaluation 
$$M:=\bM((\gs, (E)))$$
 together with a \emph{stratification} (the idea is similar to classical crystalline crystal theory). Note
$$\baroprism[1/p]((\gs, (E))) =(\gs/E)[1/p]=K,$$
hence $M$ is nothing but a $K$-vector space.
In fact (cf. \S \ref{s41ht} for details), Min-Wang show: after choosing a basis $\underline{e}$ of $M$, a stratification can be written as
\begin{equation}  \label{eqanconverge}
 \varepsilon(\underline e) = \underline e\cdot \sum_{n\geq 0}A_n \frac{X^n}{n!},  
 \end{equation} 
 where $X$ is a variable, and each $A_n$  a matrix over $K$.
Furthermore, $A_0=I$ is identity and
\begin{equation}\label{eqanplus1converge}
A_{n+1} = \prod_{i=0}^n(iE'(\pi)+A_1)  \text{ for each $n \geq 0$, and $p$-adically converges to zero as $n \to \infty$,}
\end{equation}
where $E'(\pi)$ is the evaluation at $\pi$ of the $u$-derivative $E'(u)$.
One quickly notices \eqref{eqanplus1converge} is equivalent to say that the eigenvalues of $\frac{-A_1}{E'(\pi)}$ are exactly of the form as in \eqref{eqeigen}!
Hence to show $W$ is nearly Hodge-Tate, it would suffice to show the following,  which confirms the conjecture of Min-Wang  \cite[Conj. 3.17]{MW}.


\begin{thm} \label{thmintrosen} (cf. Thm. \ref{thmMWSen}.) 
The matrix $\frac{-A_1}{E'(\pi)}$  is  the matrix  of the Sen operator   for some basis of $W$.
\end{thm}

 
The proof of Thm. \ref{thmintrosen} uses yet another ``interpretation" of the Sen operator (in addition to stratifications discussed above).
Indeed, by taking $G_L$-invariants, the space $W^{G_L}$ becomes a $\hat{L}$-vector space with $\hat{G}$-action, where $\hat{L}$ is the $p$-adic completion of $L$. Min-Wang   explicitly compute this $\hat{G}$-action.    Using the basis $\underline{e}$ from \eqref{eqanconverge}, they compute the  $\tau$-action:
$$\tau^i (\underline e) = (\underline e) (1- i  \cdot \ast)^{-\frac{A_1}{E'(\pi)}}, \quad \forall i \in \zp,$$
where $\ast$ is some constant that we do not specify here.
Thus, one observes that if we ``take the log of $\tau$", we obtain (up to scaling) our target matrix $\frac{-A_1}{E'(\pi)}$! In other words, to prove Thm. \ref{thmintrosen}, it is equivalent to show that: \emph{``taking the log of $\tau$" gives rise to Sen operator}; this will be addressed in next subsection.

\subsection{Kummer tower and Sen theory} \label{s12}
In this subsection, we explain how to use techniques of locally analytic vectors ---developed in our study of overconvergent $(\varphi, \tau)$-modules --- to prove Thm. \ref{thmintrosen}. As we explained above, the goal, roughly speaking, is to show Kummer tower can be used to recover Sen theory, and in particular, ``$\log \tau$" can recover Sen operator.

Let us first quickly recall some ideas and results in \cite{GP} to set up the context. Recall in (algebraic) $p$-adic Hodge theory, we study representations of the Galois group $G_K$. A key idea   is to first restrict the representations to some subgroups of $G_K$. A most convenient subgroup is $G_\kpinfty$ since it is normal and the quotient group $\Gamma_K$ is a $1$-dimensional \emph{$p$-adic Lie group}, to which many techniques in $p$-adic analysis --- e.g. \emph{locally analytic vectors} (cf. Def. \ref{defLAV}) --- can be applied.

Indeed, for $W\in \rep_\gk(C)$, the classical Sen module, reviewed above Def. \ref{defnnht},  can be obtained via
\begin{equation}
D_{\Sen, \kpinfty}(W)=(W^{G_\kpinfty})^{\Gamma_K\dla}
\end{equation}
Here, the notation ``$\Gamma_K\dla$" denotes the subset of locally analytic vectors under the action of  the $p$-adic Lie group $\gammak$. Using this framework, the Sen operator is precisely the \emph{Lie algebra operator}. See \S \ref{subseccycSen} for more details.

For another example, for $V\in \rep_\gk(\qp)$, the famous theorem of Cherbonnier-Colmez \cite{CC98}  says that one can attach an \emph{overconvergent $(\varphi, \Gamma)$-module} to it. To save space, we recall the following (loose) formula which constructs the \emph{rigid-overconvergent} $(\varphi, \Gamma)$-module (cf. \S \ref{subsecwtbi} for Berger's well-known ring $\wtb_{\rig}^\dagger$):
\begin{equation}\label{eqkpinftyd}
D_{\rig, \kpinfty}^\dagger(V)  \approx \left((V\otimes_\qp \wtb_{\rig}^\dagger)^{G_\kpinfty}\right)^{\Gamma_K\dla}.
\end{equation}
 (For the interested reader, the correct formula  replaces left hand side by the union  $\cup_n \varphi^{-n}(D_{\rig, \kpinfty}^\dagger(V) )$, cf. \cite[Thm. 8.1]{Ber16};  applying Kedlaya's slope filtration theorem \cite{Ked05} to these objects then recovers Cherbonnier-Colmez's theorem).

Now, consider the subgroup $G_\kinfty$ of $G_K$ constructed using the Kummer tower. The field $\kinfty$ (similar to $\kpinfty$) is still an APF extension (cf. \cite{Win83}), hence Caruso \cite{Car13} shows one can   construct the \'etale $(\varphi, \tau)$-modules (similar to  the \'etale $(\varphi, \Gamma)$-modules) to classify Galois representations.
In two papers joint with Liu and Poyeton respectively, \cite{GLAMJ, GP}, we prove a conjecture of Caruso, which says that similar to Cherbonnier-Colmez's theorem, one can also attach \emph{overconvergent $(\varphi, \tau)$-modules} to $V\in \rep_\gk(\qp)$. The proof in \cite{GP} uses crucially the idea of locally analytic vectors, indeed, the main formula (again loosely written) is the following.
\begin{equation}\label{eqkinftyd}
D_{\rig, \kinfty}^\dagger(V) \approx \left((V\otimes_\qp \wtb_{\rig}^\dagger)^{G_L}\right)^{\gamma=1, \tau\dla}.
\end{equation}
(Similar comments below \eqref{eqkpinftyd} still apply, cf. \cite{GP} for full details.) 
Here $\gamma=1$ denotes the invariant under $\gal(L/\kinfty)$-action, and  $\tau\dla$ denotes the locally analytic vectors under the  $\gal(L/\kpinfty)$-action, cf. Notation \ref{notataula}.
A key difference from \eqref{eqkpinftyd} is that we \emph{can not} take $G_\kinfty$-invariants first, because $G_\kinfty \subset G_K$ is not normal. This fact complicates the computation for \eqref{eqkinftyd}. Nonetheless, we still manage to show it is a free module of full rank, and furthermore, there is a ``comparison"  \emph{identification} (cf. Convention \ref{conv:identification}):
\begin{equation}
D_{\rig, \kinfty}^\dagger(V) \otimes \left((  \wtb_{\rig}^\dagger)^{G_L}\right)^{\hat{G}\dla}
= D_{\rig, \kpinfty}^\dagger(V) \otimes \left((  \wtb_{\rig}^\dagger)^{G_L}\right)^{\hat{G}\dla}.
\end{equation}

Now, let us come back to our objective to construct a Sen theory over the Kummer tower. It turns out the formulae are in parallel with those in the study of overconvergent $(\varphi, \tau)$-modules.

\begin{theorem} (cf. Thm. \ref{thm331kummersenmod}.)
For $W\in \rep_\gk(C)$, define
 \begin{equation}
 D_{\Sen, \kinfty}(W):= (W^{G_L})^{\gamma=1,\tau\dla}.
 \end{equation}
Then it is a $\kinfty$-vector space of   dimension equal to $\dim_C W$. In addition, there is an identification:
\begin{equation} \label{eqcompasen}
 D_{\Sen, \kinfty}(W) \otimes_\kinfty (\hat{L})^{\hat{G}\dla} = D_{\Sen, \kpinfty}(W) \otimes_\kpinfty (\hat{L})^{\hat{G}\dla} .
 \end{equation}
\end{theorem}

Let us mention that a key ingredient in the proof of above theorem is an explicit description of the set of locally analytic vectors $(\hat{L})^{\hat{G}\dla}$: one such description is already obtained in \cite{BC16};  we will obtain an alternative description which is convenient for us, cf. Prop. \ref{loc ana in L new}.
Our task does not stop here yet. Recall our objective, Thm. \ref{thmintrosen}, is to construct a  Sen theory over the Kummer tower which (up to conjugation) should be ``the same" as the cyclotomic-Sen theory.
Note since the $\kpinfty$-Sen module $D_{\Sen, \kpinfty}(W)$ (resp. the $\kinfty$-Sen module $D_{\Sen, \kinfty}(W)$) is defined via locally analytic vectors, we can define the \emph{Lie algebra operator}  $\nabla_\gamma$ associated to $\Gamma_K$-action (resp. $\nabla_\tau$ associated to $\tau$-action).
The following lemma is crucial, and follows from a   theorem of Berger-Colmez \cite{BC16}.

\begin{lemma} \label{lemlindep} (cf. Cor. \ref{corkill})
The two Lie algebra operators $\nabla_\gamma$ and  $\nabla_\tau$ are \emph{linearly dependent} over  $(\hat{L})^{\hat{G}\dla}$.
\end{lemma}  

The above lemma implies that,  roughly speaking, the two operators  $\nabla_\gamma$ and  $\nabla_\tau$ (up to scaling) acting on the comparison isomorphism \eqref{eqcompasen}, are indeed a \emph{same}  operator acting on two different basis of a common vector space! 
Since the  $\nabla_\gamma$-operator on $D_{\Sen, \kpinfty}(W)$ is precisely the Sen operator, thus so is  $\nabla_\tau$ up to scaling!
This finally achieves our desired goal.
In precise words, we obtain the following.

\begin{thm} (cf. Thm. \ref{thmkummersenop}.)
Let $W \in \rep_\gk(C)$, then we can construct a $\kinfty$-linear \emph{Sen operator over the Kummer tower}
\begin{equation}\label{eqnnablanorm}
 \frac{\nabla_\tau}{\ast} : D_{\Sen, \kinfty}(W) \to D_{\Sen, \kinfty}(W),
\end{equation}
where $\ast \in (\hat{L})^{\hat G \dla}$ is some ``normalizing constant" that we do not specify here.
This operator, after $C$-linearly extending to $\frac{\nabla_\tau}{\ast}: W \to W$, is the \emph{same} as the $C$-linear extension of classical Sen operator; hence in particular, they have the same characteristic polynomial and the same semi-simplicity property. 
\end{thm}

\begin{rem}
For the overconvergent $(\varphi, \Gamma)$-module \eqref{eqkpinftyd}, there is also a differential operator $\nabla_\gamma$ defined in \cite{Ber02}. A  differential operator   $N_\nabla$ (cf. \ref{defndiffwtb}) for the  overconvergent $(\varphi, \tau)$-modules is defined in \cite{Gaojems}, which roughly speaking is a  {normalization} of $\nabla_\tau$: this operator   plays a significant role (that can not be obtained using only $(\varphi, \Gamma)$-module theory) in  \cite{Gaojems} to develop  \emph{integral} $p$-adic Hodge theory. The above discussions, particularly Lem. \ref{lemlindep} seem to suggest that one might discover more connections between these two operators; this could be useful for integral $p$-adic Hodge theory.
\end{rem}

\subsection{Structure of the paper.}
In \S \ref{seclav}, we review some results from \cite{GP, Gaojems} on locally analytic vectors; in particular, we review some differential operators defined in \emph{loc. cit.}.
In \S \ref{seckummersen}, we construct  Sen theory over the Kummer tower, using  techniques reviewed in \S \ref{seclav}.
In \S \ref{sHT}, we use Sen theory over the Kummer tower to study Hodge-Tate prismatic crystals; we prove the conjecture of Min-Wang, and establish the link with nearly Hodge-Tate representations.

\subsection{Notations and conventions}

\begin{Notation} \label{nota hatG}
Note $\hat{G}=\gal(L/K)$ is a $p$-adic Lie group of dimension 2. We recall the structure of this group in the following.
\begin{enumerate}
\item Recall that:
\begin{itemize}
\item if $K_{\infty} \cap K_{p^\infty}=K$ (always valid when $p>2$, cf. \cite[Lem. 5.1.2]{Liu08}), then $\gal(L/K_{p^\infty})$ and $\gal(L/K_{\infty})$ generate $\hat{G}$;
\item if $K_{\infty} \cap K_{p^\infty} \supsetneq K$, then necessarily $p=2$, and $K_{\infty} \cap K_{p^\infty}=K(\pi_1)$ (cf. \cite[Prop. 4.1.5]{Liu10}) and $\pm i \notin K(\pi_1)$, and hence $\gal(L/K_{p^\infty})$ and $\gal(L/K_{\infty})$   generate an open subgroup  of $\hat{G}$ of index $2$.
\end{itemize}

\item Note that:
\begin{itemize}
\item $\gal(L/K_{p^\infty}) \simeq \Zp$, and let
$\tau \in \gal(L/K_{p^\infty})$ be \emph{the} topological generator such that
\begin{equation} \label{eq1tau}
\begin{cases} 
\tau(\pi_i)=\pi_i\mu_i, \forall i \geq 1, &  \text{if }  \Kinfty \cap \Kpinfty=K; \\
\tau(\pi_i)=\pi_i\mu_{i-1}=\pi_i\mu_{i}^2, \forall i \geq 2, & \text{if }  \Kinfty \cap \Kpinfty=K(\pi_1).
\end{cases}
\end{equation}
 
\item $\gal(L/K_{\infty})$ ($\subset \gal(K_{p^\infty}/K) \subset \Zp^\times$) is not necessarily pro-cyclic when $p=2$; however, this issue will \emph{never} bother us.
\end{itemize}
\item (This will only be briefly used in \S \ref{seclavL}).
If we let $\Delta \subset \gal(L/K_{\infty})$ be the torsion subgroup, then $\gal(L/K_{\infty})/\Delta$  is pro-cyclic; choose $\gamma' \in \gal(L/K_{\infty})$ such that its image in $\gal(L/K_{\infty})/\Delta$ is a topological generator. 
Let $\hat{G}_n \subset \hat{G}$ be the subgroup topologically generated by $\tau^{p^n}$ and $(\gamma')^{p^n}$.
\end{enumerate}
\end{Notation}

\begin{convention}
In our set-up, the Hodge-Tate-Sen weight  of   the cyclotomic character is $1$, which is the same as in \cite{Sen80} and \cite{MW}, but is opposite to that in \cite{Gaojems}. 
\end{convention}

 \begin{defn}
 Suppose $\mathcal G$ is a topological group that acts continuously on a topological ring $R$. We use $\rep_{\mathcal G}(R)$ to denote the category where an object is a finite free $R$-module $M$ (topologized via the topology on $R$) with a continuous and \emph{semi-linear} $\mathcal G$-action in the usual sense that
$$g(rx)=g(r)g(x), \forall g\in \mathcal G, r \in R, x\in M.$$
(The only case in this paper where the action is \emph{linear} is when $R=\zp$).
 \end{defn}
Examples of the above definition include $\rep_C(\gk)$ (already mentioned in the introduction), as well as several categories in \S \ref{seckummersen}.

\begin{convention}\label{conv:identification}
Let $M$ be a $A$-module where $A$ is a ring.   Let $B\subset A$ be a subring, and let $N \subset M$ be a sub-$B$-module.
If the natural map $N\otimes_B A \to M$ is an isomorphism, then we call it an \emph{identification}, and simply write
\[ N\otimes_B A \xrightarrow{\simeq} M \]
or just
\[N\otimes_B A = M\]
\end{convention}


\subsection{Acknowledgement} 
I thank 
Heng Du,
Tong Liu,
 Zeyu Liu, 
 Yu Min,
and  Yupeng Wang 
for useful discussions and correspondences. 
Special thanks to Yupeng Wang for answering my many questions and  for several useful suggestions.
The author is partially supported by the National Natural Science Foundation of China under agreement No. NSFC-12071201.

\section{Locally analytic vectors}\label{seclav}
 
 In this section, we review some results from \cite{GP, Gaojems}. We recall the notion of locally analytic vectors, and define some differential operators; these operators will be  \emph{specialized} to the Sen theory setting in \S \ref{seckummersen} and will be crucially used there.

\subsection{Locally analytic vectors}

Let us very quickly recall the theory of locally analytic vectors, see \cite[\S 2.1]{BC16} and \cite[\S 2]{Ber16} for more details.  
Recall the multi-index notations: if $\cbf = (c_1, \hdots,c_d)$ and $\kbf = (k_1,\hdots,k_d) \in \mathbb{N}^d$ (here $\mathbb{N}=\mathbb{Z}^{\geq 0}$), then we let $\cbf^\kbf = c_1^{k_1} \cdot \ldots \cdot c_d^{k_d}$. Recall that a $\Qp$-Banach space $W$ is a $\Qp$-vector space with a complete non-Archimedean  norm $\|\cdot\|$ such that $\|aw\|=\|a\|_p\|w\|, \forall a \in \Qp, w \in W$, where $\|a\|_p$ is the   $p$-adic norm on $\Qp$.

\begin{defn}\label{defLAV}
\begin{enumerate}
\item 
Let $G$ be a $p$-adic Lie group, and let $(W, \|\cdot \|)$ be a $\Qp$-Banach representation of $G$.
Let $H$ be an open subgroup of $G$ such that there exist coordinates $c_1,\hdots,c_d : H \to \Zp$ giving rise to an analytic bijection $\cbf : H \to \Zp^d$.
 We say that an element $w \in W$ is an $H$-analytic vector if there exists a sequence $\{w_\kbf\}_{\kbf \in \mathbb{N}^d}$ with $w_\kbf \to 0$ in $W$, such that $$g(w) = \sum_{\kbf \in \mathbb{N}^d} \cbf(g)^\kbf w_\kbf, \quad \forall g \in H.$$
Let $W^{H\dan}$ denote the space of $H$-analytic vectors.

\item $W^{H\dan}$ injects into $\mathcal{C}^{\an}(H, W)$ (the space of analytic functions on $H$ valued in $W$), and we endow it with the induced norm, which we denote as $\|\cdot\|_H$. We have $\|w\|_H=\sup_{\kbf \in \mathbb{N}^d}\|w_{\kbf}\|$, and $W^{H\dan}$ is a Banach space.

\item 
We say that a vector $w \in W$ is \emph{locally analytic} if there exists an open subgroup $H$ as above such that $w \in W^{H\dan}$. Let $W^{G\dla}$ denote the space of such vectors. We have $W^{G\dla} = \cup_{H} W^{H\dan}$ where $H$ runs through  open subgroups of $G$. We can endow $W^{\la}$ with the inductive limit topology, so that $W^{\la}$ is an LB space.

\item We can naturally extend these definitions to the case when $W$ is a Fr\'echet- or LF- representation of $G$, and use $W^{G\dpa}$ to denote the \emph{pro-analytic} vectors, cf. \cite[\S 2]{Ber16}.
\end{enumerate}
\end{defn}

\begin{Notation} \label{notataula}
We set up some notations with respect to representations of $\hat{G}$ (cf. Notation \ref{nota hatG}).
\begin{enumerate}
\item Given a $\hat{G}$-representation $W$, we use
$$W^{\tau=1}, \quad W^{\gamma=1}$$
to mean $$ W^{\gal(L/K_{p^\infty})=1}, \quad
W^{\gal(L/K_{\infty})=1}.$$
And we use
$$
W^{\tau\dla},   \quad  W^{\gamma\dla} $$
to mean
$$
W^{\gal(L/K_{p^\infty})\dla}, \quad  
W^{\gal(L/K_{\infty})\dla}.  $$

\item  Let
$W^{\tau\dla, \gamma=1}:= W^{\tau\dla} \cap W^{\gamma=1},$
then by \cite[Lem. 3.2.4]{GP}
$$ W^{\tau\dla, \gamma=1} \subset  W^{\hat{G}\dla}. $$
\end{enumerate}
\end{Notation}

\begin{rem}
Note that we never define $\gamma$ to be an element of $\gal(L/K_\infty)$; although when $p>2$ (or in general, when $\gal(L/K_\infty)$ is pro-cyclic), we could have defined it as a topological generator of $\gal(L/K_\infty)$. In particular, although ``$\gamma=1$" might be slightly ambiguous (but only when $p=2$), we use the notation for brevity.
\end{rem}

\begin{notation} \label{notalieg}
For $g\in \hat{G}$, let $\log g$ denote  the (formally written) series $(-1)\cdot \sum_{k \geq 1} (1-g)^k/k$. Given a $\hat{G}$-locally analytic representation $W$, the following two Lie-algebra operators (acting on $W$) are well defined:
\begin{itemize}
\item  for $g\in \gal(L/\kinfty)$ enough close to 1, one can define $\nabla_\gamma := \frac{\log g}{\log(\chi_p(g))}$;
\item for $n \gg 0$ hence $\tau^{p^n}$ enough close to 1, one can define $\nabla_\tau :=\frac{\log(\tau^{p^n})}{p^n}$.
\end{itemize}
Clearly, these two Lie-algebra operators form a $\zp$-basis of $\Lie(\hat{G}$).
\end{notation}

\subsection{Locally analytic vectors in $\wtb^I$}\label{subsecwtbi}

We briefly recall the rings $\wt{\mathbf{A}}^{I}$ and  $\wt{\mathbf{B}}^{I}$, see \cite[\S 2]{GP} for detailed discussions, also see \cite[\S 2]{Gaojems} for a faster summary.
 
Recall in Notation \ref{notaprism}, using the sequence $\pi_n$, we defined an embedding $\gs \into \ainf$, and henceforth, we  identify  $u$  with the element $[\underline{\pi}] \in \ainf$.
For $n \geq 0$, let $r_n: =(p-1)p^{n-1}$.
Let $\wt{\mathbf{A}}^{[r_\ell, r_k]}$ be the $p$-adic completion of $ \wt{\mathbf{A}}^+ [\frac{p}{u^{ep^\ell}} , \frac{u^{ep^k}}{p}]$, and let 
$$\wt{\mathbf{B}}^{[r_\ell, r_k]}: =\wt{\mathbf{A}}^{[r_\ell, r_k]}[1/p].$$
These spaces are equipped with $p$-adic topology. When $I  \subset J$ are two closed intervals as above, then by \cite[Lem. 2.5]{Ber02}, there exists a natural (continuous) embedding
$\wt{\mathbf{B}}^{J}   \hookrightarrow  \wt{\mathbf{B}}^{I}$. Hence we can define the nested intersection
$$\wt{\mathbf{B}}^{[r_\ell, +\infty)}: = \bigcap_{k \geq \ell} \wt{\mathbf{B}}^{[r_\ell, r_k]},$$
and equip it with the natural Fr\'echet  topology.  Finally, let 
$$\wt{\mathbf{B}}_{  \rig}^{\dagger}: = \bigcup_{n \geq 0} \wt{\mathbf{B}}^{[r_n, +\infty)},$$
 which is a LF space.

\begin{convention}
When $Y$ is a ring with a $G_K$-action, $X \subset \overline{K}$ is a subfield, we use $Y_X$ to denote the $\gal(\overline{K}/X)$-invariants of  $Y$.   Examples include when $Y=\wt{\mathbf{A}}^{I}, \wt{\mathbf{B}}^{I}$ and $X=L, K_\infty$. This ``style of notation" imitates that of \cite{Ber02}, which uses the subscript $\ast_{K}$ to denote $G_{\kpinfty}$-invariants.
\end{convention}

\begin{defn} \label{defnfkt}
(cf. \cite[\S 5.1]{GP} for full details).
The compatible sequence $(1, \mu_1, \mu_2, \cdots)$ in Notation \ref{notaprism} defines an element $\underline \varepsilon \in \ocflat$. Let $[\underline \varepsilon] \in \ainf$ be its Techm\"uller lift.
Let $t=\log([\underline \varepsilon])  \in \bcrisplus$ be the usual element, where $\bcrisplus$ is the usual crystalline period ring.
Define the element
\[
\lambda :=\prod_{n \geq 0} (\varphi^n(\frac{E(u)}{E(0)}))  \in \bcrisplus.\]
Define
$$ \mathfrak{t} = \frac{t}{p\lambda},$$
then it turns out $\mathfrak{t} \in \ainf$.
  \end{defn}

\begin{lemma} \label{lem b}
\cite[Lem. 5.1.1]{GP}
There exists some $n=n(\fkt) \geq 0$, such that $\mathfrak{t}, 1/\mathfrak{t} \in 
  \wt{\mathbf{B}}^{[r_n, +\infty)}$. In addition, $\mathfrak{t}, 1/\mathfrak{t} \in
 (\wt{\mathbf{B}}^{[r_n, +\infty)}_{ L})^{\hat{G}\dpa}$.
\end{lemma}

Let us caution that $\mathfrak t$ is an element of (the Banach space) $(\ainf)^{G_L}$, but it is \emph{not} a  {locally analytic vector} inside it; roughly speaking, we need the bigger spaces  $\wt{\mathbf{B}}^{I}_L$  (which contains $t$, the $p$-adic $2\pi i$) to take ``derivatives".

\begin{defn} \label{defndiffwtb}
We define two differential operators on the ring $(\wt{\mathbf{B}}_{  \rig, L}^{\dagger})^{\hat{G}\dpa}$, which are ``normalized" operators of those in Notation \ref{notalieg}. 
\begin{enumerate}
\item (cf. \cite[\S 4]{Gaojems}). Define
$$N_\nabla: (\wt{\mathbf{B}}_{  \rig, L}^{\dagger})^{\hat{G}\dpa} \to (\wt{\mathbf{B}}_{  \rig, L}^{\dagger})^{\hat{G}\dpa}$$ 
by setting
\begin{equation}\label{eqnnring}
{N_\nabla:=}
\begin{cases} 
\frac{1}{p\mathfrak{t}}\cdot \nabla_\tau, &  \text{if }  \Kinfty \cap \Kpinfty=K; \\
& \\
\frac{1}{p^2\mathfrak{t}}\cdot \nabla_\tau=\frac{1}{4\mathfrak{t}}\cdot \nabla_\tau, & \text{if }  \Kinfty \cap \Kpinfty=K(\pi_1), \text{ cf. Notation \ref{nota hatG}. }
\end{cases}
\end{equation}
 (the definition is valid over the bigger ``log ring" $(\wt{\mathbf{B}}_{  \log, L}^{\dagger})^{\hat{G}\dpa}$ which is not needed here.)
 Note that $1/\mathfrak t$ is in $ (\wt{\mathbf{B}}^\dagger_{\rig, L})^{\hat{G}\dla}$ by Lem \ref{lem b}, hence  division by $\fkt$ is allowed.
 A convenient and useful fact is that $N_\nabla$ commutes with $\gal(L/\kinfty)$, i.e., $gN_\nabla=N_\nabla g, \forall g\in \gal(L/\kinfty)$, cf. \cite[Eqn. (4.2.5)]{Gaojems}.
 
\item (cf.   \cite[5.3.4]{GP}.)
Define 
$$\partial_{\gamma}: (\wt{\mathbf{B}}_{  \rig, L}^{\dagger})^{\hat{G}\dpa} \to (\wt{\mathbf{B}}_{  \rig, L}^{\dagger})^{\hat{G}\dpa}$$ 
via
$$\partial_{\gamma}:=\frac{1}{\mathfrak t}\nabla_{\gamma}.$$
Since $\gamma(\mathfrak t) = \chi(\gamma) \cdot \mathfrak t$, we have $\nabla_{\gamma}(\mathfrak t) =\mathfrak t $ and hence
\[
\partial_{\gamma}(\mathfrak t)  = 1.
\]
\end{enumerate}
\end{defn}

\begin{rem}\label{remcompaKis}
We mention some remarks about  $N_\nabla$ that are not used in the sequel.
\begin{enumerate}
\item The   $p$ (resp. $p^2$) in the denominator of \eqref{eqnnring} makes  our monodromy operator compatible with earlier theory of Kisin in \cite{Kis06}, but \emph{up to a minus sign}. See also \cite[1.4.6]{Gaojems} for general convention of minus signs in that paper.

\item The operator $N_\nabla$ in fact restricts to an operator
$$
N_\nabla: \B_{\rig, \kinfty}^\dagger \to \B_{\rig, \kinfty}^\dagger, 
$$
where $\B_{\rig, \kinfty}^\dagger$ is the Robba ring in the $(\varphi, \tau)$-module setting, cf. \cite[\S 4]{Gaojems}.
 
\end{enumerate}
\end{rem}

\section{Kummer tower and Sen theory}
\label{seckummersen}

In this section, we first review classical (cyclotomic) Sen theory. We then compute the set of locally analytic vectors $(\hat{L})^{\hat{G}\dla}$; this is used as a \emph{bridge} transporting the $\kpinfty$-Sen theory to the $\kinfty$-Sen theory.
To be more precise, in \S \ref{subsecKS}, we define the $\kinfty$-Sen module $D_{\Sen, \kinfty}(W)$, and define the $\kinfty$-Sen operator $\frac{1}{\theta(u\lambda')}\cdot N_\nabla$. This operator, when linearly extended over $C$ (in fact, $(\hat L)^{\hat G \dla}$ is enough), becomes the \emph{same} as the classical Sen operator.

 \subsection{Cyclotomic tower and Sen theory}\label{subseccycSen}
 Let $\hatkpinfty$ be the $p$-adic completion of $\kpinfty$.
Similar to $\rep_\gk(C)$, let  $\rep_\gammak(\kpinfty)$ resp. $\rep_\gammak(\hatkpinfty)$ denote the category of \emph{semi-linear} representations.

 \begin{theorem}
 Base change functors induce equivalences of categories
 \begin{equation*}
 \rep_\gammak(\kpinfty)\simeq \rep_\gammak(\hatkpinfty) \simeq \rep_\gk(C).
 \end{equation*}
 Here, given $W\in \rep_\gk(C)$, the corresponding object in $\rep_\gammak(\hatkpinfty)$ is $W^{G_\kpinfty}$, and the corresponding object in $\rep_\gammak(\kpinfty)$ is 
\begin{equation}\label{senlav}
D_{\Sen, \kpinfty}(W): =(W^{G_\kpinfty})^{\gammak\dla}
\end{equation} 
 \end{theorem}
 \begin{proof}
 This is proved in \cite{Sen80}, except the last formula \eqref{senlav}. In \cite{Sen80}, $D_{\Sen, \kpinfty}(W)$ is recovered as the \emph{``$K$-finite vectors"}; it turns out they coincide with the \emph{locally analytic vectors}, by \cite[Thm. 3.2]{BC16}.
 \end{proof}
 
 \begin{notation}\label{notaSenop}
 Let $W\in \rep_\gk(C)$.
 \begin{enumerate}
 \item Since \eqref{senlav} implies $\Gamma_K$-action on $D_{\Sen, \kpinfty}(W)$ is locally analytic, thus the operator $\nabla_\gamma$ in Notation \ref{notalieg}   induces a operator
\begin{equation}\label{eqsenclassical}
\nabla_\gamma: D_{\Sen, \kpinfty}(W) \to D_{\Sen, \kpinfty}(W).
\end{equation}
This is called the \emph{Sen operator}: it is $\kpinfty$-linear because $\nabla_\gamma$ kills $\kpinfty$.

\item  We can \emph{$C$-linearly extend} $\nabla_\gamma$ to $D_{\Sen, \kpinfty}(W) \otimes_\kpinfty C=W$. That is, we obtain a $C$-linear operator
\begin{equation}\label{eqsenextendtoc}
\nabla_\gamma: W \to W;
\end{equation}
we still call it the \emph{Sen operator}.
 \end{enumerate}
 \end{notation}

\subsection{Locally analytic vectors in $\hat{L}$} \label{seclavL}

Let $\bdrplus$  denote the usual de Rham period ring. Let $\theta: \bdrplus \to C$ be the usual map which extends $\theta: \ainf \to \O_C$.
Recall that as in \cite[\S 2.2]{Ber02}, when $r_n \in I$, there exists a continuous embedding $\iota_n: \wt{\mathbf{B}}^{I} \hookrightarrow \bdrplus$.

\begin{lemma}
Consider the image of $\mathfrak{t}$  via the map $\theta: \ainf\to \oc$, then $0 \neq \theta(\mathfrak t) \in (\hat{L})^{\hat{G}\dla}$.  In addition, $1/\theta(\mathfrak t) \in (\hat{L})^{\hat{G}\dla}$.
\end{lemma}
\begin{proof}
We first check $\theta(\fkt)\neq 0$. Recall $\fkt=\frac{t}{p\lambda}$. Note $\theta(\frac{t}{[\underline{\varepsilon}]-1})=1$ using the expansion $t=\log([\varepsilon])$. Hence it suffices to show $\theta(\frac{[\underline{\varepsilon}]-1}{E(u)})\neq 0$: this holds because both $E$ and $\frac{[\underline{\varepsilon}]-1}{\varphi^{-1}([\underline{\varepsilon}]-1)}$ generate the principal ideal $\ker\theta$.

The proof of analyticity for $\theta(\mathfrak t)$ and $1/\theta(\mathfrak t)$ are the same; alternatively, we can use the fact that  $(\hat{L})^{\hat{G}\dla}$ is a field \cite[Lem. 2.5]{BC16}. We treat $\theta(\mathfrak t)$  in the following.
Choose $n \geq n(\fkt)$ as in Lem. \ref{lem b}  so that $\fkt \in (\wtb^{[r_n, r_n]}_L)^{\hat{G}\dla}$. 
Consider the image of $\fkt$ under the following composite map
\begin{equation}
\label{eqcompftk}
\wtb^{[r_n, r_n]} \xrightarrow{\iota_n} \bdrplus \xrightarrow{\theta} C;
\end{equation}
since both maps are continuous, hence the image is an element in $(\hat{L})^{\hat{G}\dla}$. Unfortunately, the map $\iota_n$   factors as
$$\wtb^{[r_n, r_n]} \xrightarrow{\varphi^{-n}} \wtb^{[r_0, r_0]}\xrightarrow{\iota_0} \bdrplus,$$
hence the   image of $\fkt$ under \eqref{eqcompftk} is only $\theta(\varphi^{-n}(\fkt))$. That is, we obtained
$$\theta(\varphi^{-n}(\fkt)) \in (\hat{L})^{\hat{G}\dla}.$$
Nonetheless, we have $\varphi(\fkt) =\frac{pE(u)}{E(0)} \fkt$.  One can deduce that
$$\fkt = \varphi^{-n}(\fkt) \cdot \prod_{i=1}^n \varphi^{-i}(\frac{pE(u)}{E(0)}),$$
which holds as an equality inside $\ainf$. To see that $\theta(\fkt) \in (\hat{L})^{\hat{G}\dla}$, it then suffices to see that each $\theta(\varphi^{-i}(\frac{pE(u)}{E(0)}) )$ is   locally analytic: but each of these is an element of $\kinfty$ hence is locally analytic (indeed, locally trivial). 
\end{proof}

 \begin{defn}\label{defnablaforL}
 Over the field $(\hat{L})^{\hat{G}\dla}$, we can define two differential operators, which are precisely \emph{``$\theta$-specializations"} of those in Def. \ref{defndiffwtb}.
 \begin{enumerate}
 \item Define $$N_\nabla: (\hat{L})^{\hat{G}\dla} \to (\hat{L})^{\hat{G}\dla}$$ 
by setting
\begin{equation}\label{eqnnring}
{N_\nabla:=}
\begin{cases} 
\frac{1}{p\theta(\mathfrak{t})}\cdot \nabla_\tau, &  \text{if }  \Kinfty \cap \Kpinfty=K; \\
& \\
\frac{1}{p^2\theta(\mathfrak{t})}\cdot \nabla_\tau=\frac{1}{4\theta(\mathfrak{t})}\cdot \nabla_\tau, & \text{if }  \Kinfty \cap \Kpinfty=K(\pi_1), \text{ cf. Notation \ref{nota hatG}. }
\end{cases}
\end{equation}
Similar as in Def. \ref{defndiffwtb}, we have $gN_\nabla=N_\nabla g, \forall g\in \gal(L/\kinfty)$.

\item Define 
$$\partial_{\gamma}: (\hat{L})^{\hat{G}\dla} \to (\hat{L})^{\hat{G}\dla}$$ 
via
$$\partial_{\gamma}:=\frac{1}{\theta(\mathfrak t)}\nabla_{\gamma}.$$

\item Both these (normalized) differential operators are well-defined for a $(\hat{L})^{\hat{G}\dla}$-vector space  equipped with semi-linear and locally analytic action  by $\hat{G}$.
 \end{enumerate}
 \end{defn}

The following theorem of Berger-Colmez is crucial for the discussion in the following.

\begin{theorem}\label{thmBC61}
Let $\wt{K}/K$ be a Galois extension contained in $\overline{K}$ whose Galois group is a $p$-adic Lie group with    Lie algebra $\mathfrak g$, and let $\wh{\wt K}$ be the $p$-adic completion.
There exists some $m \in \mathbb N$, a non-zero element $\mathfrak a \in \O_{\wh{\wt K}(\mu_{p^m})} \otimes_\zp \mathfrak g$ such that $\mathfrak a=0$ over $(\wh{\wt K})^{\gal(\wt{K}/K)\dla}$.
\end{theorem}
\begin{proof}
This follows from \cite[Thm. 6.1, Prop. 6.3]{BC16}. Note in \emph{loc. cit.}, one can make $\mathfrak a$ ``primitive" (defined above \cite[Thm. 6.1]{BC16}. In addition, this $\mathfrak a$ can be chosen as a certain ``Sen operator" (see \cite[Prop. 6.3]{BC16}).
\end{proof}

\begin{cor}\label{corkill}
Up to a nonzero scalar, the combination $
\theta(u\lambda') \nabla_\gamma +N_\nabla$ (from Def. \ref{defnablaforL}) is the \emph{unique} non-zero operator in $\hat{L} \otimes_\zp \Lie(\hat G)$ that kills all of $(\hat{L})^{\hat{G}\dla}$. Here $\lambda'$ is the $u$-derivative of $\lambda$ defined in Def. \ref{defnfkt}.
\end{cor}
\begin{proof}
The existence of a linear combination  $a\nabla_\gamma +b\nabla_\tau$ that kills $(\hat{L})^{\hat{G}\dla}$ is guaranteed by Thm. \ref{thmBC61}. In addition, neither $\nabla_\gamma$ nor $\nabla_\tau$ alone can kill all of $(\hat{L})^{\hat{G}\dla}$; hence the combination has to be unique up to  a non-zero scalar. 
It hence suffices to compute this operator against the element $\theta(\mathfrak t)\in  (\hat{L})^{\hat{G}\dla}$.
Indeed, we can even make the computation inside $(\wt{\mathbf{B}}_{  \rig, L}^{\dagger})^{\hat{G}\dpa}$. 
It is easy to see  $\nabla_\gamma(\mathfrak t)=\mathfrak t$. 
Using the formula in \cite[Lem. 4.1.2]{Gaojems} (which holds uniformly even when $\Kinfty \cap \Kpinfty=K(\pi_1)$), one computes
$$ N_\nabla(\mathfrak t) = N_\nabla(\frac{t}{p\lambda}) = \frac{t}{p}\cdot (-\frac{1}{\lambda^2}) N_\nabla(\lambda) = \frac{-t}{p\lambda^2}\cdot   \lambda u\cdot \frac{d}{du}(\lambda) = -\mathfrak t u\lambda'.$$
Hence we can conclude. 
We  remark that it is more convenient to use $N_\nabla$ instead of $\nabla_\tau$ in the formula as it already subsumes the   possible normalization issues  when $p=2$.
\end{proof}


 We now determine the structure of $(\hat{L})^{\hat{G}\dla}$. We first review a description Prop. \ref{loc ana in L} by Berger-Colmez.
  We then obtain an alternative description Prop. \ref{loc ana in L new} which is more convenient for us.

\begin{construction}\label{consalphan}
\begin{enumerate}
\item As in \cite[\S 4.4]{BC16},
consider the 2-dimensional $\Qp$-representation of $G_K$ (associated to our choice of $\{\pi_n\}_{n \geq 0}$) such that $g \mapsto \smat{\chi(g) & c(g) \\ 0 & 1}$ where $\chi$ is the $p$-adic cyclotomic character. Since the co-cycle $c(g)$ becomes trivial over $C_p$, there exists $\alpha \in C_p$ (indeed, $\alpha \in \hat{L}$) such that $c(g) = g(\alpha)\chi(g)-\alpha$.
This implies  $g(\alpha) = \alpha/\chi(g) + c(g)/\chi(g)$ and so $\alpha \in \hat{L}^{\hat{G}\dla}$.

\item 
Now similarly as in the beginning of \cite[\S 4.2]{BC16}, let $\alpha_n \in L$ such that $\|\alpha-\alpha_n\|_p \leq p^{-n}$. Then there exists $r(n) \gg0$ such that if $m \geq r(n)$, then $\|\alpha-\alpha_n\|_{\hat{G}_m}= \|\alpha-\alpha_n\|_p$ and $\alpha-\alpha_n \in \hat{L}^{\hat{G}_m\dan}$ (see Notation \ref{nota hatG} for $\hat{G}_m$, and see Def. \ref{defLAV} for  $\|\cdot \|_{\hat{G}_m}$ ). We can furthermore suppose that $\{r(n)\}_n$ is an increasing sequence.
\end{enumerate}
\end{construction}

\begin{defn}
Let $(H, \|\cdot \|)$ be a $\Qp$-Banach algebra such that $\|\cdot \|$ is sub-multiplicative, and let $W \subset H$ be a $\Qp$-subalgebra. Let $T$ be a variable, and let  $W \dacc{T}_n$ be the vector space consisting of $\sum_{k \geq 0} a_k T^k$ with $a_k \in W$, and $p^{nk} a_k \to 0$ when $k \to +\infty$. For $h \in H$ such that $\|h \|\leq p^{-n}$, denote $W \dacc{h}_n$ the image of the evaluation map $W \dacc{T}_n \to H$ where $T \mapsto h$.
\end{defn}

\begin{prop} 
\label{loc ana in L}
\begin{enumerate}
\item $\hat{L}^{\hat{G}\dla} =\cup_{n \geq 1} K({\mu_{r(n)}, \pi_{r(n)}})\dacc{ \alpha-\alpha_n }_n. $
\item $\hat{L}^{\hat{G}\dla, \nabla_\gamma=0} = L.$
\item $\hat{L}^{\tau\dla, \gamma=1} = {K_{\infty}}.$
\end{enumerate}
\end{prop}
\begin{proof}
Item (1) is \cite[Prop. 4.12]{BC16}, the rest follow easily, cf. \cite[Prop. 3.3.2]{GP}.
We quickly recall the proof of Item (1) here. Suppose $x\in \hat{L}^{\hat{G}_n\dan}$. For $i \geq 0$, let
$$y_i = \sum_{k \geq 0} (-1)^k (\alpha - \alpha_n)^k \nabla_\tau^{k+i}(x) \binom{k+i}{k},$$
then there exists $m\geq n$ such that $y_i \in \hat{L}^{\hat{G}_m\dan}$, and
\begin{equation} \label{eqxalpha}
x = \sum_{i \geq 0} y_i (\alpha - \alpha_n)^i \in \hat{L}^{\hat{G}_m\dan}
\end{equation}
Note  roughly speaking, \eqref{eqxalpha} is the ``Taylor expansion" of $x$ with respect to the ``variable" $\alpha - \alpha_n$. The equality  \eqref{eqxalpha} holds precisely because 
\begin{equation} \label{eqnablaalpha}
\nabla_\tau (\alpha - \alpha_n) =\nabla_\tau(\alpha)=1.
\end{equation}
Finally, the fact $\nabla_\tau(y_i)=0$ will imply that $y_i\in K(\mu_m, \pi_m)$, concluding (1).
\end{proof}

\begin{prop} \label{loc ana in L new}
 Denote $\beta = \theta(\mathfrak t)$. Apply the same procedure as in Item (2) of Construction \ref{consalphan}, choose the analogous elements $\beta_n \in L$.
Then
$$\hat{L}^{\hat{G}\dla} =\cup_{n \geq 1} K({\mu_{r(n)}, \pi_{r(n)}})\dacc{ \beta-\beta_n }_n. $$
\end{prop}
\begin{proof}
Recall   in Def. \ref{defnablaforL}, we defined $\partial_\gamma: =\frac{\nabla_\gamma}{\beta}$. Since $\nabla_\gamma(\beta)=\beta$, we have 
\begin{equation} \label{eqnablabeta}
\partial_\gamma (\beta - \beta_n) =\partial_\gamma (\beta)=1.
\end{equation}
This is the key analogue of Eqn. \eqref{eqnablaalpha}.
Now similar to the proof in Prop. \ref{loc ana in L}, suppose $x\in \hat{L}^{\hat{G}_n\dan}$, we can define
$$ z_i = \sum_{k \geq 0} (-1)^k (\beta - \beta_n)^k \partial_\gamma ^{k+i}(x) \binom{k+i}{k},$$
then there exists $m\geq n$ such that $z_i \in \hat{L}^{\hat{G}_m\dan}$, and
\begin{equation} \label{eqxbeta}
x = \sum_{i \geq 0} z_i (\beta - \beta_n)^i \in \hat{L}^{\hat{G}_m\dan}
\end{equation}
Finally, $\partial_\gamma(z_i)=0$ implies that $\nabla_\gamma(z_i)=0$ and hence $z_i\in K(\mu_m, \pi_m)$.
\end{proof}


\subsection{Kummer tower and Sen theory} \label{subsecKS}

 \begin{theorem}\label{thm331kummersenmod}
 Given  $W\in \rep_\gk(C)$ of dimension $d$, define
 \begin{equation*}
 D_{\Sen, \kinfty}(W):= (W^{G_L})^{\tau\dla, \gamma=1}.
 \end{equation*}
 Then this is a $\kinfty$-vector space of dimension $d$, and there are \emph{identifications} (cf. Convention \ref{conv:identification}):
 \begin{equation}\label{eqkpk}
 D_{\Sen, \kinfty}(W) \otimes_\kinfty (\hat{L})^{\hat{G}\dla} = D_{\Sen, \kpinfty}(W) \otimes_\kpinfty (\hat{L})^{\hat{G}\dla}=(W^{G_L})^{\hat{G}\dla}.
 \end{equation}
 \end{theorem}
 \begin{proof}
We study $D_{\Sen, \kinfty}(W)= (W^{G_L})^{\tau\dla, \gamma=1}$ in two steps. In Step 1, we show $(W^{G_L})^{\hat{G}\dla, \nabla_\gamma=0}$ is of dimension $d$ over $L$, this is achieved via a monodromy descent. In Step 2, via an (easy) \'etale descent, we further show the ($\gamma=1$)-invariant $D_{\Sen, \kinfty}(W)$ has dimension $d$. 
 
Step 1 (monodromy descent). We claim the following:
\begin{itemize}
\item Let $M$ be a $\hat{L}^{\hat{G}\dla}$-vector space of dimension $d$ with a semi-linear and locally analytic $\hat{G}$-action. Then the subspace $M^{\nabla_\gamma=0}$ is a   $L$-vector space of dimension $d$ such that
\begin{equation}
M^{\nabla_\gamma=0} \otimes_L \hat{L}^{\hat{G}\dla} =M
\end{equation}
\end{itemize}
 This is a ``$\theta$-specialization" of the argument \cite[Rem. 6.1.7]{GP}, hence the proof is practically verbatim. 
 To proceed, as in Prop. \ref{loc ana in L new}, we denote $\beta=\theta(\mathfrak{t})$. There, we also made use of the operator
  $\partial_\gamma =\frac{1}{\beta}  \nabla_\gamma$, which is precisely $\theta$-specialization  of the operator (with same notation) in \cite[5.3.4]{GP}.
 Choose a basis of $M$, and let  $D_\gamma=\Mat(\partial_\gamma)$, then it suffices to show that there exists $H\in \GL_d(\hat{L}^{\hat{G}\dla})$ such that
 \begin{equation}\label{eqdgamma}
 \partial_{\gamma}(H)+D_{\gamma}H = 0
 \end{equation} 
For $k \in \mathbb N$, let $D_k = \Mat(\partial_{\gamma}^k)$. For $n$ large enough, the series given by
$$H = \sum_{k \geq 0}(-1)^kD_k\frac{(\beta-\beta_n)^k}{k!}$$
converges   to the desired  solution of \eqref{eqdgamma}. Here, $\beta-\beta_n$ is used as a ``variable" just as in the proof of Prop. \ref{loc ana in L new}.
  
  Step 2 (etale descent). 
  By \cite[Prop. 3.1.6]{GP}, we know  
  \begin{equation}\label{eq316first}
  (W^{G_L})^{\hat{G}\dla} =D_{\Sen, \kpinfty}(W) \otimes_\kpinfty (\hat{L})^{\hat{G}\dla}.
  \end{equation}
 Apply Step 1 to the above vector space, and  so 
 $$X:=(W^{G_L})^{\hat{G}\dla, \nabla_\gamma=0}$$
  is a vector space over $L$ of  dimension $d$. In addition, $X$ is stable under $\gal(L/\kinfty)$-action as this action commutes with $\nabla_\gamma$. (Note however $\tau$-action does not commute with $\nabla_\gamma$, not even on the ring level: for example $\tau \nabla_\gamma(u)=0 \neq \nabla_\gamma \tau(u)$.) 
  In summary, $X$ is a $L$-vector space with a $\gal(L/\kinfty)$-action. Note $L=\cup_n K(\pi_n, \mu_n)$ and $\gal(L/\kinfty)$ is topologically finitely generated. Thus, for $n\gg 0$, $X$ descends to some $\gal(L/\kinfty)$-stable  vector space  $X_n$ over $K(\pi_n, \mu_n)$. 
 By Galois descent,   $X_n^{\gamma=1}$ is a $K(\pi_n)$-vector space of dimension $d$, and hence $X_n^{\gamma=1}\otimes_{K(\pi_n)} \kinfty$ is precisely the  desired $D_{\Sen, \kinfty}(W)$. Finally, apply \cite[Prop. 3.1.6]{GP} again, then we have
 \begin{equation}
 (W^{G_L})^{\hat{G}\dla} =D_{\Sen, \kinfty}(W) \otimes_\kinfty (\hat{L})^{\hat{G}\dla},
 \end{equation}
 which together with \eqref{eq316first} proves \eqref{eqkpk}. 
 \end{proof}

Let $W \in \rep_\gk(C)$, since $D_{\Sen, \kinfty}(W) $ are locally analytic vectors, we can define (cf. Def. \ref{defnablaforL})
\begin{equation}\label{eq322tau}
N_\nabla: D_{\Sen, \kinfty}(W)  \to (W^{G_L})^{\hat{G}\dla}
\end{equation}

\begin{theorem}\label{thmkummersenop}
Let $W \in \rep_\gk(C)$, then Eqn. \eqref{eq322tau}, after linear scaling,  induces a $\kinfty$-linear operator, which we call the \emph{Sen operator over the Kummer tower}
\begin{equation}\label{eqnnablanorm}
\frac{1}{\theta(u\lambda')}\cdot N_\nabla: D_{\Sen, \kinfty}(W) \to D_{\Sen, \kinfty}(W).
\end{equation}
(We also sometimes use the simplified terminology ``$\kinfty$-Sen operator").
Extend it $C$-linearly to a $C$-linear operator on $ D_{\Sen, \kinfty}(W) \otimes_\kinfty C=W$, and denote it by the same notation:
\begin{equation}
\frac{1}{\theta(u\lambda')}\cdot N_\nabla:  W \to W
\end{equation}
Then this is precisely the (uniquely defined) \emph{Sen operator} in Eqn. \eqref{eqsenextendtoc}.
\end{theorem} 

 \begin{proof}
 Note we have the relation  $gN_\nabla=N_\nabla g, \forall g\in \gal(L/\kinfty)$, hence $N_\nabla$ stabilizes $D_{\Sen, \kinfty}(W)$. 
 Furthermore, $\theta(u\lambda') \in \kinfty$, hence $\frac{1}{\theta(u\lambda')}\cdot N_\nabla$ still stabilizes $D_{\Sen, \kinfty}(W)$. Thus \eqref{eqnnablanorm} is well-defined.
 This is a $\kinfty$-linear operator  because $\nabla_\tau$ hence $\frac{1}{\theta(u\lambda')}\cdot N_\nabla$ kills $\kinfty$.
 
By Cor. \ref{corkill}, 
$\mathfrak{a}= \nabla_\gamma +\frac{1}{\theta(u\lambda')} N_\nabla$ is an $(\hat{L})^{\hat{G}\dla}$-\emph{linear} operator  on both sides of the following:
  \begin{equation*}
 D_{\Sen, \kinfty}(W) \otimes_\kinfty (\hat{L})^{\hat{G}\dla} =D_{\Sen, \kpinfty}(W) \otimes_\kpinfty (\hat{L})^{\hat{G}\dla}.
 \end{equation*}
 On the right hand side,  $N_\nabla$  kills $ D_{\Sen, \kpinfty}(W) $ and hence $\mathfrak{a} =\nabla_\gamma$ is precisely the $(\hat{L})^{\hat{G}\dla}$-linear extension of the Sen operator in Eqn. \eqref{eqsenclassical}. Similarly, on the left hand side, $\nabla_\gamma$ kills $ D_{\Sen, \kinfty}(W)$, and hence $\mathfrak{a}$ is the same as the $(\hat{L})^{\hat{G}\dla}$-linear extension of \eqref{eqnnablanorm}. We can  conclude by further extending $C$-linearly.
 \end{proof}

\begin{cor}
The two operators $\frac{1}{\theta(u\lambda')}\cdot N_\nabla$ in Eqn. \eqref{eqnnablanorm} and $\nabla_\gamma$ in Eqn. \eqref{eqsenclassical} have the same eigenvalues, and have the same semi-simplicity property.
\end{cor}

\begin{notation}
Let $F$ be a field. Write $\mathrm{Endo}_F$ for the category consisting of $(M, f)$ where $M$ is a
   finite dimensional $F$-vector space and $f: M\to M$ is an $F$-linear endomorphism.
\end{notation}  
  
  \begin{prop} \label{propisomobj}
Let $W_1, W_2 \in \rep_\gk(C)$. 
By abuse of notation, we use $\phi_\mathrm{Sen}$ to denote the Sen operators on $D_{\Sen, \kpinfty}(W_i)$ as well as on $W_i$, cf. Notation \ref{notaSenop}.
Similarly, we use   $\phi_{\kinfty-\mathrm{Sen}}$ to denote the $\kinfty$-Sen operators on $D_{\Sen, \kinfty}(W_i)$, cf. Thm. \ref{thmkummersenop}.
The following are equivalent.
\begin{enumerate}
\item $W_1$ and $W_2$ are isomorphic as objects in $\rep_\gk(C)$.
\item   $(D_{\Sen, \kpinfty}(W_1), \phi_\mathrm{Sen})$ and $(D_{\Sen, \kpinfty}(W_2), \phi_\mathrm{Sen})$ are isomorphic as objects in $\mathrm{Endo}_\kpinfty$.
\item   $(W_1,  \phi_\mathrm{Sen})$ and $(W_2,  \phi_\mathrm{Sen})$ are isomorphic as objects in $\mathrm{Endo}_C$.
\item    $(D_{\Sen, \kinfty}(W_1), \phi_{\kinfty-\mathrm{Sen}})$ and $(D_{\Sen, \kinfty}(W_2), \phi_{\kinfty-\mathrm{Sen}})$  are isomorphic as objects in $\mathrm{Endo}_\kinfty$.
\end{enumerate}
\end{prop}
\begin{proof}
Clearly, (1) implies (2)-(4). (2) obviously implies (3); (4) implies (3) by  Thm. \ref{thmkummersenop}.
The implication from (3) to (1) is proved in \cite[p. 101, Thm. 7]{Sen80}.
\end{proof}


\section{Hodge-Tate prismatic crystals} \label{sHT}
 
In \S \ref{s41ht} and \S \ref{subsec42}, we   review results of Min-Wang  about Hodge-Tate prismatic crystals and their relation with  stratifications. In \S \ref{subsecHTsen}, we show that a linear operator appearing in the stratification, upon scaling, is precisely our \emph{$\kinfty$-Sen operator} constructed in \S \ref{subsecKS}. This makes it possible to classify rational Hodge-Tate prismatic crystals by nearly Hodge-Tate representations.
In \S \ref{subsec44}, we make comparison with a theorem of Sen.

 \subsection{Hodge-Tate prismatic crystals}\label{s41ht}
 
  \begin{defn}[Hodge--Tate crystals]\label{Dfn-Hodge--Tate crystal} \cite[Def. 3.1]{MW}
  \begin{enumerate}
  \item   A \emph{Hodge-Tate prismatic crystal} on $(\calO_K)_{\Prism}$ is a sheaf $\bM$ of $\overline \calO_{\Prism}$-modules such that for any $(A,I)\in(\calO_K)_{\Prism}$, $\bM((A,I))$ is a finite projective $A/I$-module and that for any morphism $(A,I)\to (B,J)$, the natural map
    \[\bM((A,I))\otimes_A B\to \bM((B,J)).\]
      is an isomorphism.
    Denote by $\Vect((\calO_K)_{\Prism},\overline \calO_{\Prism})$ the category of such objects.
    
    \item The category $\Vect((\calO_K)_{\Prism},\overline \calO_{\Prism}[1/p])$ is defined analogously using the sheaf $\calO_{\Prism}[1/p]$: e.g., $\bM((A,I))$ is a finite projective $(A/I)[1/p]$-module.
  \end{enumerate}
  \end{defn}

 \begin{defn}[Stratification]\label{Dfn-stratification} cf. \cite[Def. 3.14]{MT} or \cite[Def. 3.2]{MW}.
    Let $(B,I) \in \okprism$, and let  $(B^{\bullet},IB^{\bullet})$ be  the associated co-simplicial prism. Let $p_0, p_1: B \to B^1$, $p_{01}, p_{02}, p_{12}: B^1 \to B^2$ be the canonical morphisms.
      Let $M$ be a $B$-module. 
\begin{enumerate}
\item A \emph{stratification} on $M$ is a $B^1$-linear isomorphism \[\varepsilon: M\otimes_{B,p_1}B^1\to M\otimes_{B,p_0}B^1.\]
\item Say a stratification $\varepsilon$ satisfies the  \emph{cocycle condition} if 
$$p^*_{01}(\varepsilon)p^*_{12}(\varepsilon) = p_{02}^*(\varepsilon).$$
\end{enumerate}      
  \end{defn}

The Breuil-Kisin prism $(\gs, E)$ is a cover of the final object of $\mathrm{Shv}((\ok)_\prism)$ (cf. \cite[Lem. 2.2]{MW}), hence by a well-known argument, e.g., \cite[Cor. 3.9]{MT}, the category of Hodge-Tate prismatic crystals   is equivalent to the category of finite free $\calO_K$-modules equipped with stratifications (with respect to the prism $(\gs, E)$) satisfying the  cocycle condition.
  
To compute the stratification explicitly, one  needs to know the structures of the rings $\gs^i/E$.
  Note for any $i\geq 0$, we have
  \[\gs^{i+1} = W(k)[[u_0,\dots,u_i]]\{\frac{u_0-u_1}{E(u_0)},\dots,\frac{u_0-u_i}{E(u_0)}\}_{\delta}^{\wedge_{(p,E(u_0))}},\]
which is the $(p, E(u_0))$-adic completion of the $\delta$-ring $W(k)[[u_0,\dots,u_i]]\{\frac{u_0-u_1}{E(u_0)},\dots,\frac{u_0-u_i}{E(u_0)}\}_{\delta}$. This ring is extremely complicated;  fortunately, it becomes very simple after modulo $E$.

\begin{lem}\label{pd polynomial} \cite[Lem. 2.7]{MW}
  For any $1\leq i\leq n$, denote $X_i$ the image of $\frac{u_0-u_i}{E(u_0)}\in \gs^{n}$ modulo $(E)$, then 
  $$\gs^n/(E) \simeq \calO_K\{X_1,\dots, X_n\}^{\wedge}_{\rm pd}$$
  where the right hand side is the free pd-polynomial ring on the variables $X_1,\dots, X_n$.   
  \end{lem}

Thus, for $M$   a finite free $\calO_K$-module, a stratification with respect to $(\gs, E)$ can be written explicitly. Indeed, for $\underline{e}=(e_1,\dots,e_l)$   an $\calO_K$-basis of $M$,  a stratification can be expressed via
 \begin{equation} 
 \varepsilon(\underline e) = \underline e\cdot \sum_{n\geq 0}A_n\frac{X^n}{n!},
 \end{equation} 
  where  each $A_n$ is a matrix over $\ok$.
We now determine all such possible $A_n$.

\begin{prop}\label{matrix of stratification} cf. \cite[Thm. 3.5, Rem. 3.7]{MW}
Let $M$ be  a finite free $\calO_K$-module with a basis $\underline e$. Then a rule 
 \begin{equation} \label{eqstrat}
 \varepsilon(\underline e) = \underline e\cdot \sum_{n\geq 0}A_n\frac{X^n}{n!} 
 \end{equation} 
  determines a stratification satisfying the cocycle condition if and only if 
\begin{itemize}
\item $A_0 = I$, and 
\item $A_{n+1} = \prod_{i=0}^n(iE'(\pi)+A_1)$ for each $n \geq 0$, and converges to zero as $n \to \infty$.
\end{itemize}
If the above is satisfied, Eqn. \eqref{eqstrat} can be re-written as
\begin{equation}\label{eq413rewritten}
\varepsilon(\underline e) = \underline e\cdot (1-E'(\pi)X)^{-\frac{A_1}{E'(\pi)}}
\end{equation}
where $E'(\pi)$ is the evaluation at $\pi$ of the derivative $E'(u)$.
\end{prop}

\begin{defn}\label{defHTmod}
Define the (tensor) category of Hodge-Tate prismatic \emph{modules}  to consist of  pairs $(M, f)$ where $M$ is a finite  free $\calO_K$-module, $f: M \to M$ is an $\ok$-module endomorphisms such that all of its eigenvalues (in $\overline{K}$) are of the form
\begin{equation}\label{eqeigen415}
iE'(\pi) +\fkm_{\O_{\overline{K}}} \text{ for some } i \in \mathbb Z.
\end{equation}
 That is, the eigenvalues belong to $\mathbb Z +\fkm_{\O_{\overline{K}}}$ when $K$ is unramified, and belong to $\fkm_{\O_{\overline{K}}}$ when $K$ has non-trivial ramification.
 A morphism $(M, f) \to (N, g)$ is an $\ok$-module homomorphism $M \to N$ compatible with $f$ and $g$.
\end{defn}

\begin{cor}\label{cor416}
The category $\Vect((\calO_K)_{\Prism},\overline \calO_{\Prism})$ is tensor equivalent to the category of Hodge-Tate prismatic modules via the evaluation functor:
$$\bM \mapsto \bM((\gs, (E))).$$
\end{cor} 
\begin{proof}
Let $(M, f)$ be an $\ok$-module with an endomorphism, and let $A_1$ be the matrix of $f$ with respect to a chosen basis of $M$. These data, via Eqn. \eqref{eq413rewritten}, induce a stratification if and only if $\prod_{i=0}^n(iE'(\pi)+A_1)$ converges to zero. This is equivalent to the condition that for each eigenvalue $\alpha$ of $A_1$, the element $\prod_{i=0}^n(iE'(\pi)+\alpha)$ converges to zero in $C$. Hence in particular, \emph{one} of the terms $iE'(\pi)+\alpha$ is in the maximal ideal, giving rise to the condition \eqref{eqeigen415}.
Conversely, if $iE'(\pi)+\alpha$ is in the maximal ideal, then each $jE'(\pi)+\alpha$ is integral, and each $(i+kp)E'(\pi)+\alpha$ is in the maximal ideal, and hence $\prod_{i=0}^n(iE'(\pi)+\alpha)$ converges to zero.
\end{proof}
 
 \begin{rem}\label{remratHT}
 All the above discussions hold for \emph{rational} Hodge-Tate prismatic crystals, simply by replacing all appearances of ``$\ok$-modules" by ``$K$-vector spaces".
 \end{rem}
 
\subsection{Hodge-Tate crystals on perfect prismatic site}\label{subsec42}


\begin{defn} \label{Dfn-rational Hodge--Tate crystal}
    A \emph {rational Hodge-Tate prismatic crystal} on the perfect prismatic site $(\calO_K)^{\perf}_{\Prism}$ is a sheaf $\bM$ of $\overline \calO_{\Prism}[\frac{1}{p}]$-modules such that for any $(A,I)\in(\calO_K)^{\perf}_{\Prism}$, $\bM((A,I))$ is a finite projective $A/I[\frac{1}{p}]$-module and that for any morphism $(A,I)\to (B,J)$, the natural map
    \[\bM((A,I))\otimes_A B  \to \bM((B,J))\]
    is an isomorphism.
    Denote by $\Vect((\calO_K)^{\perf}_{\Prism},\overline \calO_{\Prism}[\frac{1}{p}])$ the category of such objects.
\end{defn}

\begin{construction} Let $\bM \in  \Vect((\calO_K)^{\perf}_{\Prism},\overline \calO_{\Prism}[\frac{1}{p}])$. Let $\xi:=\frac{[\underline{\varepsilon}]-1}{\varphi^{-1}([\underline{\varepsilon}]-1)}$, which is a generator of $\ker\theta$.
\begin{enumerate}
\item The evaluation of $\bM$ at $(\ainf, (\xi))$ is a $C$-vector space, which has a canonical semi-linear $G_K$-action induced from the $G_K$-action on the prism $(\ainf, (\xi))$.

\item Let $\A_{\inf, L}:=(\ainf)^{G_L}$, then $(\A_{\inf, L}, (\xi))$ is still a perfect prism. The evaluation of $\bM$ at  $(\A_{\inf, L}, (\xi))$ induces a $\hat{L}$-vector space with a semi-linear $\hat{G}$-action.
\end{enumerate}
\end{construction}

\begin{prop}\label{crystal is C-rep} \cite[Thm. 3.12, Rem. 3.13]{MW}
The evaluations  above   induce  tensor equivalences of categories:
\begin{equation}
\begin{tikzcd}
                                                                                                & \rep_\gk(C)                      \\
{\Vect((\calO_K)^{\perf}_{\Prism},\overline \calO_{\Prism}[\frac{1}{p}])} \arrow[ru] \arrow[rd] &                                  \\
                                                                                                & \rep_{\hat G}(\hat L) \arrow[uu]
\end{tikzcd}
\end{equation}
Here the vertical equivalence is a well-known consequence of Falting's almost purity theorem (essentially because $\hat L$ is a perfectoid field).
\end{prop}

 \subsection{Hodge-Tate prismatic crystals and Sen theory}\label{subsecHTsen}

\begin{notation}
 \label{consrescrys}
Let  $\bM\in\Vect((\calO_K)_{\Prism},\overline \calO_{\Prism}[1/p])$.  Consider evaluations
 $$M:=\bM((\gs, E)), \quad V(\bM):=\bM((\A_{\inf, L}, (\xi))) \in \rep_{\hat{G}}(\hat{L}), \quad W:=\bM((\ainf, \xi)) \in \rep_\gk(C);  $$
 they give rise to  identifications (of vector spaces)
\begin{equation}
\label{eqdimbm}
M\otimes_{\calO_K}\hat{L} = V(\bM), \quad V(\bM)\otimes_{\hat{L}}C=W.
\end{equation}
\end{notation}

\begin{prop}\label{representation comes from crystal} \cite[Thm. 3.15]{MW}
Use notations in \ref{consrescrys}.
Let $(M, f)$ be the corresponding pair as in Def. \ref{defHTmod}, except now with $M$ a $K$-vector space, cf. Rem. \ref{remratHT}. 
Let $\underline{e}$ be a basis of $M$, and let $A_1$ be the matrix of $f$ with respect to this   basis.
 Then the (semi-linear) action of $\hat{G}$ on $V(\bM)$ is given by:
  \begin{equation} \label{eqcocyU}
      g(\underline{e})  = (\underline{e})\cdot \big(1-c(g)  E'(\pi)  \pi   \lambda_{\mathrm{MW}}  (1-\mu_p)\big)^{-\frac{A_1}{E'(\pi)}}.
  \end{equation}
 Here $c(g)\in\Zp$ is determined by $g(\underline \pi) = \underline{\varepsilon}^{c(g)}\underline \pi$; the symbol $\lambda_{\mathrm{MW}}$, denoted as ``$\lambda$" in \cite[Prop. 3.14]{MW}, equals to $\theta( {\xi}/{ E} ) $.
\end{prop}


\begin{thm}\label{thmMWSen}
Use notations in \ref{consrescrys} and Prop. \ref{representation comes from crystal}.  
\begin{enumerate}
\item We have
 \begin{equation}\label{eqMDsen}
  D_{\Sen, \kinfty}(W) =M\otimes_K  \kinfty.
  \end{equation}
  
 \item  Scaling $f: M \to M$ by $\frac{-1}{E'(\pi)}$ and extending $\kinfty$-linearly gives rise to
$$\frac{-1}{E'(\pi)}f: M\otimes_K  \kinfty \to M\otimes_K  \kinfty;$$
then this is the \emph{same} (via \eqref{eqMDsen}) as the $\kinfty$-Sen operator 
\begin{equation*} 
\frac{1}{\theta(u\lambda')}\cdot N_\nabla: D_{\Sen, \kinfty}(W) \to D_{\Sen, \kinfty}(W)
\end{equation*}
 defined in Thm. \ref{thmkummersenop}.
 \end{enumerate}
 \end{thm}

The above theorem proves \cite[Conj. 3.17]{MW}: indeed,   the $C$-linear extension
$$\frac{-1}{E'(\pi)}f: W\to W,$$
via Thm. \ref{thmkummersenop} again, is  the  Sen operator on $W$ as in Notation \ref{notaSenop}.

\begin{proof}
Item (1). The formula \eqref{eqcocyU} implies that   $\gal(L/\kinfty)$ acts trivially on $M$. In addition, the formula 
\begin{equation}\label{eqtaue433}
       \tau^i (\underline{e}) =(\underline{e}) \big(1- c(\tau^i)\cdot E'(\pi)\cdot \pi \cdot \lambda_{\mathrm{MW}} \cdot (1-\mu_p)\big)^{-\frac{A_1}{E'(\pi)}},
  \end{equation}
implies that elements of $M$ are $\tau$-locally analytic.  
Here note 
$$c(\tau^i)=i, \text{ resp. } 2i, \text{ when } \Kinfty \cap \Kpinfty=K, \text{ resp. }  K(\pi_1).$$
In summary, 
$$M \subset V(\bM)^{\gamma=1, \tau\dla}$$
  Thus \cite[Prop. 3.1.6]{GP} implies that
  \begin{equation*}
  D_{\Sen, \kinfty} =M\otimes_K  \kinfty.
  \end{equation*}
 
Item (2).  The formula \eqref{eqtaue433} further implies that
\begin{equation*} \label{mwtau}
\nabla_\tau (\underline{e}) = (\underline{e})\pi \lambda_{MW} (1-\mu_p) A_1, \quad \text{ resp. }  (\underline{e})2\pi \lambda_{MW} (1-\mu_p) A_1, 
\end{equation*}
$\text{ when } \Kinfty \cap \Kpinfty=K, \text{ resp. }  K(\pi_1)$.
This implies our $\kinfty$-Sen operator acts via
\begin{equation}
\frac{1}{\theta(u\lambda')}\cdot N_\nabla(\underline{e}) 
= (\underline{e})\frac{ \pi \lambda_{MW} (1-\mu_p)A_1  }{p\theta(u\lambda' \fkt)}
= (\underline{e}) \frac{ \pi \lambda_{MW} ( \mu_p-1) E'(\pi)}{p\theta(u\lambda' \fkt)}\cdot \frac{-A_1}{E'(\pi)}
\end{equation}
To finish the proof of this theorem, it suffices to check the ``scalar term" is $1$, namely:
\begin{equation*}
\frac{ \pi \lambda_{MW} (\mu_p-1) E'(\pi)}{p\theta(u\lambda'\fkt)}=1
\end{equation*}
Using $\lambda_{MW} =\theta(\frac{\mu}{\varphi^{-1}(\mu)E} )$ where $\mu=[\underline{\varepsilon}]-1$,
and $\theta(\varphi^{-1}(\mu))=\mu_p-1$, the above identity simplifies as 
\begin{equation*}
\theta( \frac{\mu E'}{p\lambda' \fkt E}  )=1
\end{equation*}
Note $\theta(\frac{t}{\mu})=1$ using $t=\log(\mu+1)$, it suffices to check
\begin{equation}\label{eqwitht}
\theta( \frac{t E'}{p\lambda' \fkt E}  )=1
\end{equation}
 Apply multiplication rule to $\lambda'$, one sees that
 \begin{equation*}
 \theta(\lambda') =\frac{E'(\pi)}{E(0)} \theta(\varphi(\lambda))
 \end{equation*}
Hence \eqref{eqwitht} becomes
\begin{equation*}
\theta( \frac{t }{p\frac{E}{E(0)}  \varphi(\lambda) \fkt}  )=1
\end{equation*}
 The denominator is $p\frac{E}{E(0)}  \varphi(\lambda) \fkt=p\lambda \fkt =t$, thus we can conclude.
\end{proof}

\begin{remark}
We discuss   ``canonicity" of the equivalence in Cor. \ref{cor416}.
\begin{enumerate}
 \item 
Let $\pi_1, \pi_2$ be two uniformizers of $K$  with minimal polynomials $E_1, E_2$ over $K_0$.
Thm. \ref{thmMWSen} implies that the linear endomorphisms encoded in the equivalence Cor. \ref{cor416} (with respect to the two different Breuil-Kisin prisms) are differed by a scalar of $ {E_1'(\pi_1)}/{E_2'(\pi_2)}$. 
Hence the equivalence in Cor. \ref{cor416} is independent of choices of uniformizers if and only if $K$ is unramified.  Nonetheless, see also next item.

\item We want to warn the readers a potential pitfall. In the identification $M\otimes_K C=W$ via Eqn. \eqref{eqdimbm}, $W$ is independent of  choices of uniformizers. We can regard $M$ as a  $K$-sub-vector space of $W$; this sub-space is \emph{not} independent of choices of the Kummer tower $\kinfty$, (and hence perhaps we can say, that as a subspace of $W$, $M$ is \emph{not} ``canonical".) Indeed, suppose otherwise, that $M$ is independent of choices of Kummer towers $\kinfty$. The formula   \eqref{eqcocyU} implies that  $M$ has trivial action by $\gal(\overline{K}/\kinfty)$, now for \emph{any possible} Kummer tower $\kinfty$. But these different possible   $\gal(\overline{K}/\kinfty)$'s generate the entire $\gal(\overline{K}/K)$, by \cite[Lem. F15]{EG} or the discussions in \cite[Lem. 7.3.2]{Gaojems}. That is to say, $G_K$ acts on $M$ trivially; this implies that the Sen operator has to be zero, by formula   \eqref{eqcocyU}, or by \cite[p.100, Thm. 6]{Sen80}.

\end{enumerate}
\end{remark}

The following theorem finally completes the diagram \eqref{diagnht}. Recall $\rep_\gk^{\mathrm{nHT}}(C)$ is defined in Def. \ref{defnnht}.

\begin{theorem} \label{thmnht}
The evaluation functor $\bM \mapsto W:=\bM((\ainf, \xi))$ in Notation \ref{consrescrys} induces an equivalence of categories:
$$\Vect((\calO_K)_{\Prism},\overline \calO_{\Prism}[ {1}/{p}]) \simeq \rep_\gk^{\mathrm{nHT}}(C).$$
\end{theorem} 
\begin{proof}
Let $\bM \in \Vect((\calO_K)_{\Prism},\overline \calO_{\Prism}[1/p])$, we first show $W(\bM)$ is nearly Hodge-Tate.
Let $(M, f)$ be the associated pair via Cor. \ref{cor416}, then Thm. \ref{thmMWSen} implies $\frac{-1}{E'(\pi)}f$ is the Sen operator; hence   the condition \eqref{eqeigen415} translates into the condition \eqref{eqeigen} in Def. \ref{defnnht}. This shows $W(\bM)$ is nearly Hodge-Tate, and gives our desired functor
\begin{equation}\label{434functor}
\Vect((\calO_K)_{\Prism},\overline \calO_{\Prism}[ {1}/{p}]) \to \rep_\gk^{\mathrm{nHT}}(C).
\end{equation}

We now prove the functor is essentially surjective. 
For \emph{any}  $W \in \rep_\gk(C)$, let $\phi_{\Sen}$ be the Sen operator on $D_{\Sen, \kpinfty}(W)$. 
By \cite[p.100, Thm. 5]{Sen80}, there exists a sub-$K$-vector space of full dimension $M \subset D_{\Sen, \kpinfty}(W)$ which is stable under $\phi_{\Sen}$. \footnote{Note this sub-$K$-vector space is in general not unique. For example, consider the trivial $C$-representation whose Sen operator is the zero map, then \emph{any} sub-$K$-vector space is stable under the zero map.}
If $W$ is in addition nearly Hodge-Tate, then the eigenvalues of $\phi_{\Sen}|_M$ has to satisfy the condition in   \eqref{eqeigen}, and hence  $(M, f=-E'(\pi)\cdot \phi_{\Sen}|_M)$ is a pair satisfying condition \eqref{eqeigen415}.
Via Cor. \ref{cor416}, we can obtain a stratification and hence a rational Hodge-Tate prismatic crystal $\bm$. Let $W' =\bm((\ainf, (\xi))) \in \rep_\gk(C)$ be the associated representation, one needs to check it is isomorphic to $W$. By Thm. \ref{thmMWSen},
$D_{\Sen, \kinfty}(W') =M\otimes_K \kinfty$, and its $\kinfty$-Sen operator is precisely $\frac{-1}{E'(\pi)}f= \phi_{\Sen}|_M\otimes 1$; hence in particular, by Thm. \ref{thmkummersenop}, the $C$-linear  Sen   operator on $W'=M\otimes_K C$ is the \emph{same} as that on $W=M\otimes_K C$. Thus $W'$ is isomorphic to $W$ by Prop. \ref{propisomobj}. (We point to Rem. \ref{remconfusion} for a possible confusion in the construction above.)

We now prove the functor is fully faithful.
Let $\bm_1, \bm_2 \in \Vect((\calO_K)_{\Prism},\overline \calO_{\Prism}[\frac{1}{p}])$, 
let $(M_1, f_1), (M_2, f_2)$ be the corresponding pairs  via Cor. \ref{cor416},
and let $W_1, W_2 \in \rep_\gk^{\mathrm{nHT}}(C)$ be the corresponding representations.
It suffices to show
$$\mathrm{Morph}((M_1, f_1), (M_2, f_2)) \to \mathrm{Morph}(W_1, W_2) $$
is a bijection. It is injective because $M_i$ contains a basis of $W_i$; it hence suffices to show both sides have the same $K$-dimension. 
Note that \eqref{434functor} is a tensor functor  and note $f_i$'s are (scaled) Sen operators. 
By considering the $C$-representation $\hom_C(W_1, W_2)$, it suffices to show the following statement: given $\bm \in \Vect((\calO_K)_{\Prism},\overline \calO_{\Prism}[\frac{1}{p}])$ with corresponding $(M, f)$ and $W$, then we have
\begin{equation*}
M^{f=0} =  W^{G_K};
\end{equation*}
but this follows from \cite[p.100, Thm. 6]{Sen80} (which implies the above two spaces are isomorphic after tensoring with $C$).
\end{proof}
 
 
\begin{remark}\label{remconfusion}
 Let us point out a possible confusion in the proof of essential surjectivity above.
 In the proof, we obtained a $C$-linear isomorphim $W \to W'$ compatible with the $C$-linear Sen operators via 
\begin{equation}\label{eqlinsen}
W \simeq M\otimes_K C \simeq \bm((\gs, E))\otimes_K C   \simeq W'
\end{equation}
   where the first isomorphism follows from Sen's construction and the rest follow by construction in Notation \ref{consrescrys}. Nonetheless, the ``identity" map  $M\otimes_K C \simeq \bm((\gs, E))\otimes_K C $ is \emph{not} ``$G_K$-equivariant"! 
   Indeed,  the $\tau$-action on $M$ is trivial under classical Sen theory; whereas the $\tau$-action on $\bm((\gs, E))$ is in general not trivial by Prop. \ref{representation comes from crystal}.
   The powerful aspect of Sen's theorem \cite[p. 101, Thm. 7]{Sen80}, is that the $\phi_\Sen$-equivariant map \eqref{eqlinsen} already guarantees existence of a $G_K$-equivariant isomorphism between $W$ and $W'$, although in an \emph{implicit} way!
\end{remark}

\subsection{Comparison with a theorem of Sen} \label{subsec44}
We recall a (rather peculiar) theorem of Sen which also concerns about ``$K$-rationality" of the Sen operator, and make some  comparisons   with our result.

\begin{theorem} \label{thmsenthm10}
\cite[p.110, Thm. 10]{Sen80}
Suppose the residue field $k$ is \emph{algebraically closed}. Then there exists a  (\emph{non-canonical}) equivalence  between $\rep_\gk(C)$ and the category of $K$-vector spaces equipped with a linear operator. This non-canonical equivalence, which depend on choices for each simple object in $\rep_\gk(C)$, is \emph{never} compatible with tensor products.
\end{theorem}

\begin{para}\label{remsen10}
Let us say a few words about the rather technical proof of the above theorem. 
\begin{enumerate}
\item 
Recall, as we already used in the proof of our Thm. \ref{thmnht}, Sen shows in  \cite[p.100, Thm. 5]{Sen80} that for any $C$-representation, the Sen operator is   ``$K$-rational": that is, it is stable on a (certainly non-unique!) sub-$K$-vector space.

\item 
 Conversely, when the residue field $k$ is algebraically closed, \cite[p.104, Thm. 9]{Sen80} shows that any $K$-matrix can be recovered as a Sen operator for some $C$-representation.
Note however, as stated in  \cite[p.105, Thm. 9']{Sen80}, this only constructs a   \emph{set bijection} between   $H^1(G_K, \GL_n(C))$ and the set of similarity classes of $n\times n$ matrices over $K$. One still needs to build some \emph{functoriality} into this bijection. 

\item 
Sen then \emph{chooses}, for each \emph{simple} $C$-representation $U$, a $K$-vector space $U_0$ stable under Sen operator; this builds a functor for \emph{semi-simple} objects. To get the complete functor, one first map a general $C$-representation $W$ to its semi-simple part $W_S$, then one notes $(W_S)_0$ is still stable under the Sen operator of $W$ (not just that of $W_S$!), since Sen operator is Galois equivariant.
\end{enumerate}
\end{para}

\begin{remark}\label{remfon214}
In the remark following \cite[Thm. 2.14]{Fon04}, Fontaine mentions that Thm. \ref{thmsenthm10} can also be obtained using his \emph{classification} of $C$-representations: this  also points to the  non-canonicity of Thm. \ref{thmsenthm10}.
\end{remark}

\begin{para}

In fact, in $p$-adic Hodge theory, many typical  theorems say that there is an equivalence of categories:
$$\mathrm{MOD} \simeq \mathrm{REP}$$
where $\mathrm{MOD}$ is certain ``module category" and $\mathrm{REP}$ is certain ``representation category". We invite the readers to consider the examples such as results of Bhatt-Scholze \cite{BSFcrystal}, of Cherbonnier-Colmez \cite{CC98}, or Thm. \ref{thmnht}.
In proofs of these results, it is normally easy to construct a  \emph{functor} from  $\mathrm{MOD}$ to $\mathrm{REP}$ (sometimes perhaps initially to a bigger representation category, such as in our  Thm. \ref{thmnht}). The most difficult part almost always is to show the functor is essentially surjective (or, to determine its essential image).

What happens with Sen's result Thm. \ref{thmsenthm10}, is that there is not even an obvious \emph{map} from $\mathrm{MOD}$ to $\mathrm{REP}$; 
Sen can still construct a such map, by the \emph{set bijection} mentioned in  \ref{remsen10}(2): but then this is hopeless to be a \emph{functor}. This forces Sen to go the other way around, by constructing a \emph{functor} from  $\mathrm{REP}$ to $\mathrm{MOD}$, which then unfortunately relies on many choices, and cannot be canonical. 
We regard this strong contrast with our Thm. \ref{thmnht} as another hint (in addition to  Rem. \ref{remintro115}) that the nearly Hodge-Tate representations deserve to be further studied and applied.
  \end{para}

 
 \bibliographystyle{alpha}

\begin{thebibliography}{BMS19}

\bibitem[ALB]{ALB}
Johannes Ansch\"{u}tz and Arthur-C\'{e}sar Le~Bras.
\newblock Prismatic {D}ieudonn\'{e} theory.
\newblock {\em preprint}.

\bibitem[BC16]{BC16}
Laurent Berger and Pierre Colmez.
\newblock {Th{\'e}orie de {S}en et vecteurs localement analytiques}.
\newblock {\em Ann. Sci. {\'E}c. Norm. Sup{\'e}r. (4)}, 49(4):947--970, 2016.

\bibitem[Ber02]{Ber02}
Laurent Berger.
\newblock {Repr{\'e}sentations {$p$}-adiques et {\'e}quations
  diff{\'e}rentielles}.
\newblock {\em Invent. Math.}, 148(2):219--284, 2002.

\bibitem[Ber16]{Ber16}
Laurent Berger.
\newblock {Multivariable {$(\varphi,\Gamma)$}-modules and locally analytic
  vectors}.
\newblock {\em Duke Math. J.}, 165(18):3567--3595, 2016.

\bibitem[BL22a]{BL1}
Bhargav Bhatt and Jacob Lurie.
\newblock Absolute prismatic cohomology.
\newblock {\em preprint}, 2022.

\bibitem[BL22b]{BL2}
Bhargav Bhatt and Jacob Lurie.
\newblock {T}he prismatization of $p$-adic formal schemes.
\newblock {\em preprint}, 2022.

\bibitem[BMS18]{BMS1}
Bhargav Bhatt, Matthew Morrow, and Peter Scholze.
\newblock Integral {$p$}-adic {H}odge theory.
\newblock {\em Publ. Math. Inst. Hautes \'{E}tudes Sci.}, 128:219--397, 2018.

\bibitem[BMS19]{BMS2}
Bhargav Bhatt, Matthew Morrow, and Peter Scholze.
\newblock Topological {H}ochschild homology and integral {$p$}-adic {H}odge
  theory.
\newblock {\em Publ. Math. Inst. Hautes \'{E}tudes Sci.}, 129:199--310, 2019.

\bibitem[BS19]{BSprism}
Bhargav Bhatt and Peter Scholze.
\newblock Prisms and prismatic cohomology.
\newblock {\em preprint}, 2019.

\bibitem[BS21]{BSFcrystal}
Bhargav Bhatt and Peter Scholze.
\newblock Prismatic {F}-crystals and crystalline {G}alois representations.
\newblock {\em preprint}, 2021.

\bibitem[Car13]{Car13}
Xavier Caruso.
\newblock {Repr{\'e}sentations galoisiennes {$p$}-adiques et
  {$(\varphi,\tau)$}-modules}.
\newblock {\em Duke Math. J.}, 162(13):2525--2607, 2013.

\bibitem[CC98]{CC98}
F.~Cherbonnier and P.~Colmez.
\newblock {Repr{\'e}sentations {$p$}-adiques surconvergentes}.
\newblock {\em Invent. Math.}, 133(3):581--611, 1998.

\bibitem[DL21]{DL21}
Heng Du and Tong Liu.
\newblock A prismatic approach to $(\varphi, \hat{G})$-modules and
  {$F$}-crystals.
\newblock {\em preprint}, 2021.

\bibitem[EG]{EG}
Matthew Emerton and Toby Gee.
\newblock Moduli stacks of \'etale $(\varphi, {\Gamma})$-modules and the
  existence of crystalline lifts.
\newblock {\em to appear, Annals of Math. Studies}.

\bibitem[Fon04]{Fon04}
Jean-Marc Fontaine.
\newblock Arithm\'{e}tique des repr\'{e}sentations galoisiennes {$p$}-adiques.
\newblock Number 295, pages xi, 1--115. 2004.
\newblock Cohomologies $p$-adiques et applications arithm\'{e}tiques. III.

\bibitem[Gao]{Gaojems}
Hui Gao.
\newblock {Breuil-Kisin modules and integral $p$-adic Hodge theory}.
\newblock {\em to appear, J. Eur. Math. Soc.}
\newblock With Appendix A by Y.~ Ozeki, and Appendix B by H.~ Gao and T.~ Liu.

\bibitem[GL20]{GLAMJ}
Hui Gao and Tong Liu.
\newblock Loose crystalline lifts and overconvergence of \'{e}tale
  {$(\varphi,\tau)$}-modules.
\newblock {\em Amer. J. Math.}, 142(6):1733--1770, 2020.

\bibitem[GP21]{GP}
Hui Gao and L\'{e}o Poyeton.
\newblock Locally analytic vectors and overconvergent {$(\varphi,
  \tau)$}-modules.
\newblock {\em J. Inst. Math. Jussieu}, 20(1):137--185, 2021.

\bibitem[Ked05]{Ked05}
Kiran~S. Kedlaya.
\newblock Slope filtrations revisited.
\newblock {\em Doc. Math.}, 10:447--525, 2005.

\bibitem[Kis06]{Kis06}
Mark Kisin.
\newblock {Crystalline representations and {$F$}-crystals}.
\newblock In {\em {Algebraic geometry and number theory}}, volume 253 of {\em
  {Progr. Math.}}, pages 459--496. Birkh{\"a}user Boston, Boston, MA, 2006.

\bibitem[Kos21]{Kos21}
Teruhisa Koshikawa.
\newblock Logarithmic prismatic cohomology {I}.
\newblock {\em preprint}, 2021.

\bibitem[Liu08]{Liu08}
Tong Liu.
\newblock {On lattices in semi-stable representations: a proof of a conjecture
  of {B}reuil}.
\newblock {\em Compos. Math.}, 144(1):61--88, 2008.

\bibitem[Liu10]{Liu10}
Tong Liu.
\newblock {A note on lattices in semi-stable representations}.
\newblock {\em Math. Ann.}, 346(1):117--138, 2010.

\bibitem[MT]{MT}
Matthew Morrow and Takeshi Tsuji.
\newblock Generalised representations as $q$-connections in integral $p$-adic
  {H}odge theory.
\newblock {\em preprint}.

\bibitem[MW21]{MW}
Yu~Min and Yupeng Wang.
\newblock On the {H}odge-{T}ate-crystals over $\mathcal{O}_{K}$.
\newblock {\em preprint}, 2021.

\bibitem[MW22]{MW22}
Yu~Min and Yupeng Wang.
\newblock $p$-adic {S}impson correpondence via prismatic crystals.
\newblock {\em preprint}, 2022.

\bibitem[Sen81]{Sen80}
Shankar Sen.
\newblock Continuous cohomology and {$p$}-adic {G}alois representations.
\newblock {\em Invent. Math.}, 62(1):89--116, 1980/81.

\bibitem[Ser79]{Serrelocal}
Jean-Pierre Serre.
\newblock {\em Local fields}, volume~67 of {\em Graduate Texts in Mathematics}.
\newblock Springer-Verlag, New York-Berlin, 1979.
\newblock Translated from the French by Marvin Jay Greenberg.

\bibitem[Win83]{Win83}
Jean-Pierre Wintenberger.
\newblock Le corps des normes de certaines extensions infinies de corps locaux;
  applications.
\newblock {\em Ann. Sci. \'Ecole Norm. Sup. (4)}, 16(1):59--89, 1983.

\end{thebibliography}

\end{document}